\documentclass[12pt]{article}

\usepackage[retainorgcmds]{IEEEtrantools}
\usepackage{amsmath} 
\usepackage{amssymb} 
\usepackage{amsthm} 
\usepackage{graphicx} 
\usepackage{verbatim} 
\usepackage{color} 
\usepackage{hyperref} 

\theoremstyle{definition} \newtheorem{defn}{Definition}
\theoremstyle{definition} \newtheorem{thm}{Theorem}
\theoremstyle{definition} \newtheorem{lemma}{Lemma}
\theoremstyle{definition} \newtheorem*{remark}{Remark}
\theoremstyle{plain} \newtheorem{example}{Example}
\theoremstyle{definition} 
\theoremstyle{definition} 

\begin{document}

\author{D.Fletcher}
\title{Construction of fixed points of certain substitution systems by interlacing arrays in $1$ and $2$ dimensions}
\maketitle

\begin{abstract}

This paper describes an alternative method of generating fixed points of certain substitution systems. This method centres on taking infinite words consisting of one repeated letter per word. These infinite words are then interlaced to form a new, more complex, infinite word. By considering particular limits of interlacings of words, fixed points of substitutions are generated. This method is then extended to two dimensions, where a structure equivalent to a well known aperiodic tiling (the Robinson tiling) is constructed.

\end{abstract}

\section*{Introduction}

In the first section of this paper we describe our fundamental operation, the $\cap$ operation. This operation acts on two infinite words (or $1$D arrays) to produce a new infinite word, by placing a letter of the second word between every two letters of the first word. We then show how this operation can be linked to Toeplitz sequences, and show how iterating the operation can create more complex words. In the second section of this paper we introduce a generalisation of the $\cap$ operation, namely $\cap^{n}$. This operation places a letter from the second word after the $(kn)$ th letter of the first word (for all $k \in \mathbb{Z}$). We then consider limits of these operations, and prove that carefully selected limits not only exist, but are fixed points of specific substitutions. Hence we have an alternate method of calculating fixed points of (a family of) substitutions on words.
The results in the first two sections link in to known Toeplitz sequence theory. The third section extends these results to two dimensions, detailing how this method can generate a Robinson tiling (one of the first \lq aperiodic' tilings, as described in \cite{ROBIN}). The paper ends with a brief summary of possible uses and avenues for further investigation.

This paper attempts to err on the side of clarity, instead of brevity. As such we have included several worked examples were we walk the reader through new concepts slowly.
Of course the reader is free to skip worked examples.

On a first reading we suggest that the reader briefly skims subsections $1.1$ and $1.2$ and the start of section $2$, then skips ahead to subsection $2.4$ in order to understand what the method can do in the one dimensional case. After this, section $3$ can be read to see how the method extends to two dimensions, or the skipped subsections can be read for concrete proofs.

\section{The $\cap$ operation}

The $\cap$ operation uses the concept of an \emph{$\mathbb{Z}^{d}$-array}.

We will use these arrays to describe both the $1$ and $2$ dimensional cases, using similar terminology for both dimensions. In this section we will introduce the ($1$ dimensional) $\cap$ operation, leaving the $2$ dimensional case to the third section.  
Let us consider the following definition of $\mathbb{Z}^{d}$-arrays. This formulation of $\mathbb{Z}^{d}$-arrays was first studied in \cite{DOWNAROWICZ}, and extended in \cite{CORTEZ}.


\begin{defn}[$\mathbb{Z}^{d}$-array] \label{ZDarray}
A \emph{$\mathbb{Z}^{d}$-array} (with alphabet $\Sigma$) is a function

$Z :  \mathbb{Z}^{d} \mapsto \Sigma$
\end{defn}

The precise formulation of $\mathbb{Z}^{d}$-arrays we are using is new to this thesis, but is equivalent to the version found in \cite{CORTEZ}.




\begin{defn} \label{underlyinglattice}
Choose a set of basis vectors $\{ v_{1}, \ldots , v_{d}\}$ of $\mathbb{R}^{d}$.
Define the \emph{underlying lattice} $Latt$ as $\{a_{1}v_{1} + a_{2}v_{2} + \ldots + a_{d}v_{d} | a_{1}, a_{2}, \ldots \in \mathbb{Z} \}$.
Define $p : \mathbb{Z}^{d} \mapsto Latt$ as $p(a_{1}, a_{2}, \ldots , a_{d}) = a_{1}v_{1} + a_{2}v_{2} + \ldots + a_{d}v_{d}$. This is clearly a bijective function.

Then the \emph{array on a lattice} $(Z,Latt)$ is the function;

$Z' : Latt \mapsto \Sigma$ such that $Z' \circ p  = Z$.

In other words, $Z$ is the pullback of $Z'$. Note that if you have an array $Z$ and an underlying lattice $Latt$, you can calculate $Z'$. Similarly, if you have $Z'$ you can calculate $Z$. Thus we will often refer to \lq arrays on a lattice' as \lq arrays', when the lattice is fixed.

\end{defn}

These definitions correspond to assigning an element from an alphabet $\Sigma$ to every point on a lattice embedded in $\mathbb{R}^{d}$. The underlying lattice $Latt$ encodes which lattices we choose and the $\mathbb{Z}^{d}$-array encodes what elements from the alphabet get assigned to the particular lattice points.
From now on we will use the convention that when no underlying lattice is mentioned, we are using the integer points in $\mathbb{R}^{d}$ for our underlying lattice. In one dimension this corresponds to bi-infinite sequences.

By considering the Voronoi diagram of a lattice, we can construct an associated tiling for any lattice. A tile in this case would be the Voronoi cell of a point $p$ in a lattice $A \subset \mathbb{R}^{d}$  (namely the set of all points $x \in \mathbb{R}^{d}$ such that $d(x,p) = d(x,A)$). Thus we can consider the $\cap$ operation as an operation on tilings as well.

\subsection{$1$ dimension}

The major advantage of studying the $1$D case is that we can express the array as a bi-infinite \emph{sequence}, enabling us to use results applicable to sequences.
Thus we will define the $\cap$ operation for $1$D arrays, and \emph{Toeplitz sequences}, a class of examples originating in the field of dynamical systems \cite{JACOBKEANE} \cite{Kurka}.

\begin{defn}[$A \cap B$, the Superposition operation] \label{Superpositionoperation}

Let $A$ and $B$ be $\mathbb{Z}$-arrays, namely $A: \mathbb{Z} \mapsto \Sigma$, $B: \mathbb{Z} \mapsto \Sigma$.

Then define $A \cap B$ as follows;

\begin{equation*}
(A \cap B) (v) =
\begin{cases}
B(n) & \text{if } v=2n,\\
A(n) & \text{if } v=2n-1.
\end{cases}
\end{equation*}

\end{defn}

\begin{defn}
Define $f_{X}(A)$ as the function sending an array $A$ to $A \cap X$, where $X$ is another array.
\end{defn}

\subsection{Motivation}

We can consider $A\cap B$ in one dimension as the overlaying, or merger of two separate tilings.
Imagine the tilings $A$ and $B$ as two infinite translucent physical strips with the strip representing each tiling.
The operation $A \cap B$ corresponds with placing $A$ down over the origin normally, then shifting it half a unit to the left. You then place the strip for $B$ down over the origin normally, effectively interlacing points from $A$ and $B$.
Looking through $A$ and $B$'s strips, the tiling $A \cap B$ can be read off. Note that the distance between consecutive points needs to be scaled back up to $1$ (by expanding about the origin by a factor of $2$).
See picture \ref{layers1D} for a schematic motivation of how to create $A \cap B$ from two tilings $A$ and $B$.

\begin{figure}[!hbtp]
\includegraphics[angle=90, width=0.8\textwidth]{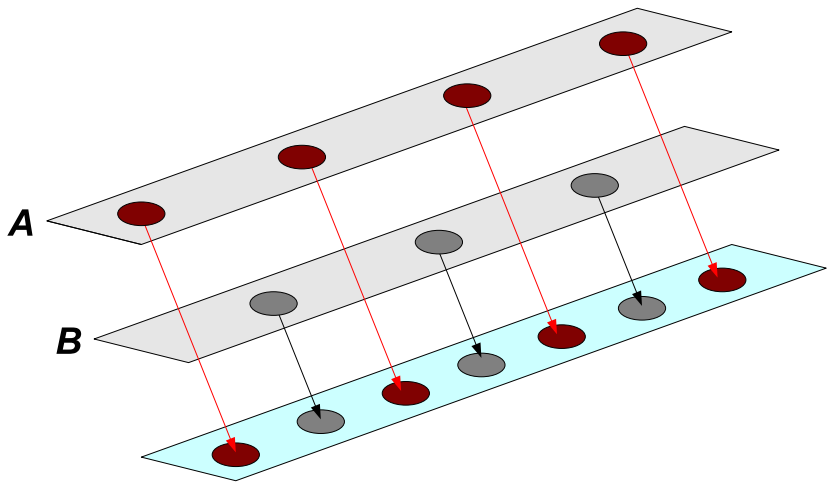}
\caption{A schematic motivation for $A \cap B$.}
\label{layers1D}
\end{figure}

We will now give the definition of a Toeplitz sequence.
We will also define a \lq null' array, a concept used in the field of Toeplitz sequences.
Note that there are many equivalent definitions of a Toeplitz sequence. We will use one more easily applicable to this thesis, from the reference \cite{JACOBKEANE} (with a very similar format to  \cite{IWANIKLACROIX}).

\begin{defn}
Let $A$ be a finite set of at least two elements.
Let $A^{\ast}$ be the set of finite sequences, or words over $A$.
If $w \in A^{\ast}$, let $|w|$ denote its length.
Let $\Omega = A^{\mathbb{Z}}$.
If $S \in \Omega$, $n \in \mathbb{Z}$ and $p \geq 1$, then let

\[S_{n} S_{n+1} \ldots S_{n+p-1}
\]
denote the word of length $p$ appearing in $S$ starting at position $n$.
Thus $S_{n}$ is the $n$th letter in the sequence.

\end{defn}

\begin{defn}
An element $S \in \Omega$ is called a \emph{periodic} sequence  with period $p \in \mathbb{N}$ if $S_{t} = S_{t+p}$ for all $t \in \mathbb{Z}$.
\end{defn}

\begin{defn}
An element $S \in \Omega$ is called a \emph{Toeplitz sequence} if it is not a periodic sequence, and satisfies the following condition;

\[\forall n \in \mathbb{Z}, \exists p \geq 2 \text{ such that } \forall k \in \mathbb{Z}, S_{n+kp} = S_{n}  \]
\end{defn}

\begin{defn}\label{periodicpart}
A \emph{p-periodic part} of a bi-infinite sequence $S \in A^{\mathbb{Z}}$ is;

$Per_{p}(S) = \{n \in \mathbb{Z} : \forall k \in \mathbb{Z}, S_{n+pk} = S_{n} \}$
\end{defn}




\begin{remark} \label{periodicparttoeplitz}
Consider a sequence $T$, which is not periodic.
If every point $T_{k} \in T$ is in a $p_{t}$-periodic part for some $p_{t}$, $T$ is a Toeplitz sequence.
This is because if a point $T_{n}$ belongs to a $p_{t}$-periodic part, then by definition of periodic part, $T_{n+kp}=T_{n}$.
\end{remark}

\begin{defn}
An \emph{almost Toeplitz sequence} is a sequence where all but a finite number of points are in a periodic part.
\end{defn}

In this chapter we will be creating (almost) Toeplitz sequences. The points not in a periodic part will be a finite word directly to the right of the origin. We will refer to this word as the \emph{seed}, for reasons which will be apparent later.

\begin{defn}
A \emph{null} array $\ast$ is the unique array with alphabet $\{ \ast \}$.
\end{defn}



Regarding the $\cap$ operation, let us define it to be an operation that is conventionally evaluated, as follows;

\begin{defn}
$A \cap B \cap C := ((A \cap B) \cap C)$
Furthermore;
$A_{1} \cap A_{2} \cap \ldots \cap A_{n+1} := (A_{1} \cap A_{2} \cap \ldots \cap A_{n}) \cap A_{n+1})$ for all $n \in \mathbb{N}$.
\end{defn}

\begin{defn}
$(A \cap B)^{2} = A \cap B \cap A \cap B$

$(f_{A} \circ f_{B})^{2}(X) = f_{A} \circ f_{B} \circ f_{A} \circ f_{B} (X) = X \cap B \cap A \cap B \cap A$
\end{defn}

\begin{defn}
If $B_{1}, B_{2}, B_{3}, \ldots$ is a sequence of arrays $\mathbb{Z} \mapsto \Sigma$, say they have the limit $B : \mathbb{Z} \mapsto \Sigma$ if $\forall n \in \mathbb{N}$ $\exists m$ such that $B(x) = B_{i}(x)$ for all $-n \leq x \leq n$ and $i \geq m$.
\end{defn}

For studying infinite $\cap$ operations, we need to be careful when looking at the limit.

Consider two arrays, $\underline{0}$ and $\underline{1}$, where $\underline{0}(v)=0$ and $\underline{1}(v) = 1$ for all $v \in \mathbb{Z}$.
Now consider the series $\{X_{i} \}_{i={1}}^{\infty} = \underline{0}, \underline{0} \cap \underline{1}, \underline{0} \cap \underline{1} \cap \underline{0}, \ldots$ as shown in figure \ref{array00cap1astcap1}.

\begin{figure}[!hbtp]
\includegraphics[angle=0, width=0.8\textwidth]{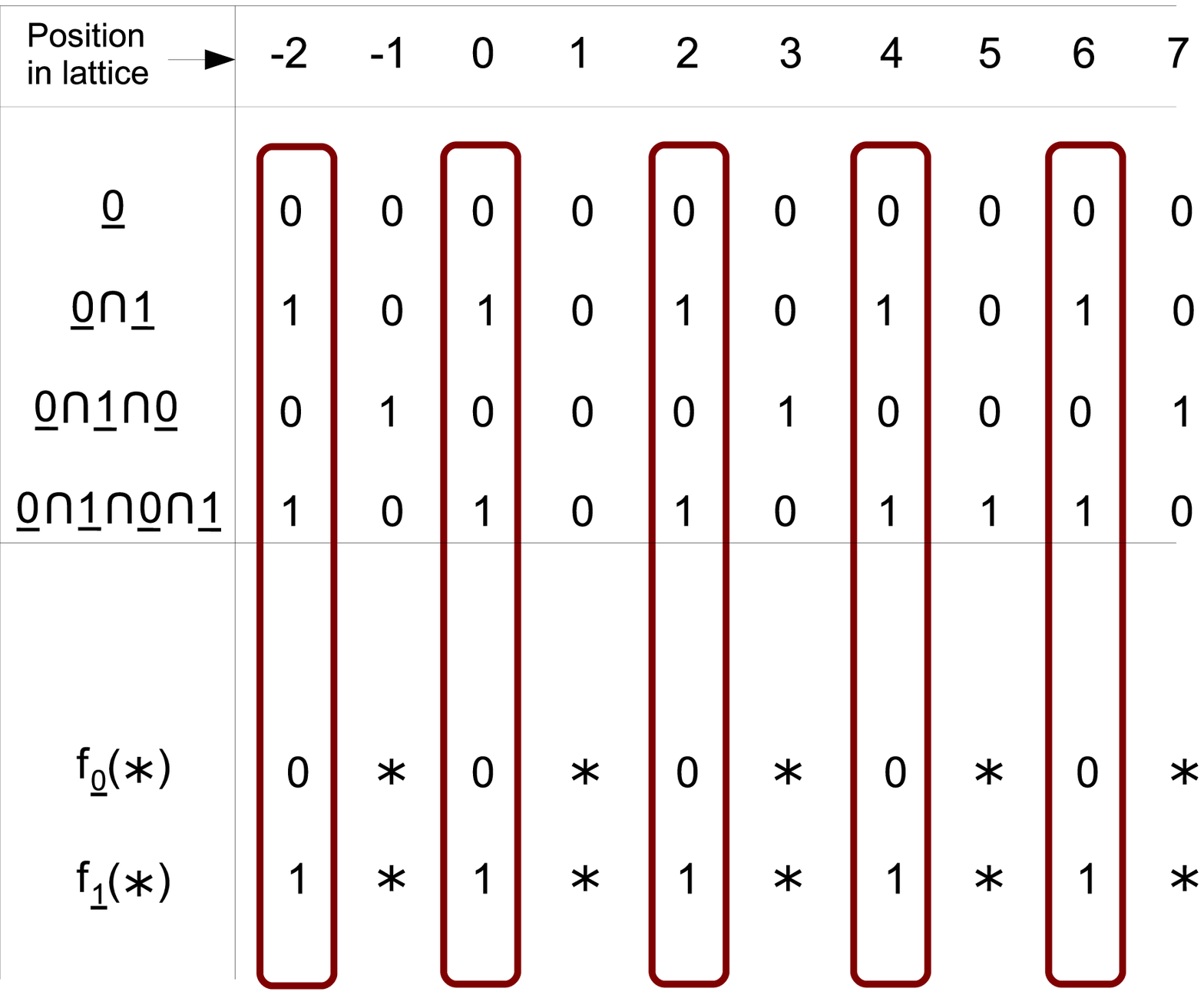}
\caption{The first few points of \(\underline{0}, \underline{0} \cap \underline{1}\) and other arrays.}
\label{array00cap1astcap1}
\end{figure}

By examining definition \ref{Superpositionoperation}, the reader can observe that points in the even positions of an array $A \cap B$ are only dependent on the array $B$. Thus the same result applies for a function $f_{B}(A)$, since it is merely different terminology.

We will use this fact to show that no limit of the series $\{X_{i} \}_{i={1}}^{\infty}$ exists.


The sequence $\{X_{i} \}_{i={1}}^{\infty} = \underline{0}, \underline{0} \cap \underline{1}, \underline{0} \cap \underline{1} \cap \underline{0}, \ldots$ can be rewritten as \[\underline{0}, f_{\underline{1}} (\underline{0}), f_{\underline{0}} (f_{\underline{1}} (\underline{0})), \ldots, f_{\underline{0}} (X_{2i+1}), f_{\underline{1}} (X_{2(i+1)}), f_{\underline{0}} (X_{2(i+1)+1)}) \ldots\]
 More precisely, $X_{2i}= f_{\underline{1}} (X_{2i-1})$, $X_{2i+1}= f_{\underline{0}} (X_{2i})$  and $X_{1}= \underline{0}$.

The value of even positions in $X_{2i}=f_{\underline{1}}(X_{2i-1})$ is $1$. The value of even positions in $X_{2i+1}$ is $0$. Thus the value of points in even positions will be different in $X_{2i}$ and $X_{2i+1}$, for any $i \in \mathbb{N}$. Thus there does not exist a limit of this sequence.




Limits of $\cap$ operations do not (usually) exist, but we can instead use the common idea of convergent subsequences, in an attempt to get interesting results.
For example with the limit of $\underline{0} \cap \underline{1} \cap \underline{0} \cap \underline{1} \cap \ldots$ we could consider the limit of

\[\underline{0} \cap \underline{1}, (\underline{0} \cap \underline{1})^{2}, (\underline{0} \cap \underline{1})^{3}, \ldots \] or
\[ \underline{0}, \underline{0} \cap \underline{1} \cap \underline{0}, (\underline{0} \cap \underline{1})^{2} \cap \underline{0}, \ldots \]

(These subsequences do have properly defined limits, which will be proven in a later section).
Considering different subsequences may give us different values for a limit. Note that the work in this section implies that only superpositions of the form $A_{1} \cap A_{2} \cap \ldots \cap A_{n} \cap A \cap A \cap A \ldots$ could have limits when considered in the naive way.

\section{$A \cap^{n} B$}

We will now define an operation without obvious equivalent in higher dimensions, $\cap^{n}$. We will then describe the limit of a sequence of $\cap^{n}$ superposition operations.

\begin{defn}
Let $A$ and $B$ be one dimensional sequences.
Define $A \cap^{n} B : \mathbb{Z} \mapsto \Sigma$ as follows;

\begin{equation*}
(A \cap^{n} B) (v) =
\begin{cases}
B(\frac{v}{n+1}) & \text{if } v=k(n+1) \text{ for } k \in \mathbb{Z} ,\\
A(v- \lfloor\frac{v}{n+1}\rfloor) & \text{otherwise}.
\end{cases}
\end{equation*}

\end{defn}

Intuitively you place $B_{0} \in B$ at the origin, bumping $A_{0}$ left one unit, and this bumping all points left of the origin one unit left.
You then add $B_{1}$ to the right of $A_{n}$, bumping tiles $A_{N}$ for $N > n$ right by $1$. Similarly you add $B_{-1}$ to the right of $A_{-n}$, bumping tiles away from the origin to make space.
You then place $B_{2}$ to the right of $A_{2n}$ and $B_{-2}$ to the right of $A_{-2n}$, and iterate for $B_{3}$ and $B_{-3}$, and so on.

\begin{figure}[!hbtp]
\includegraphics[angle=0, width=0.8\textwidth]{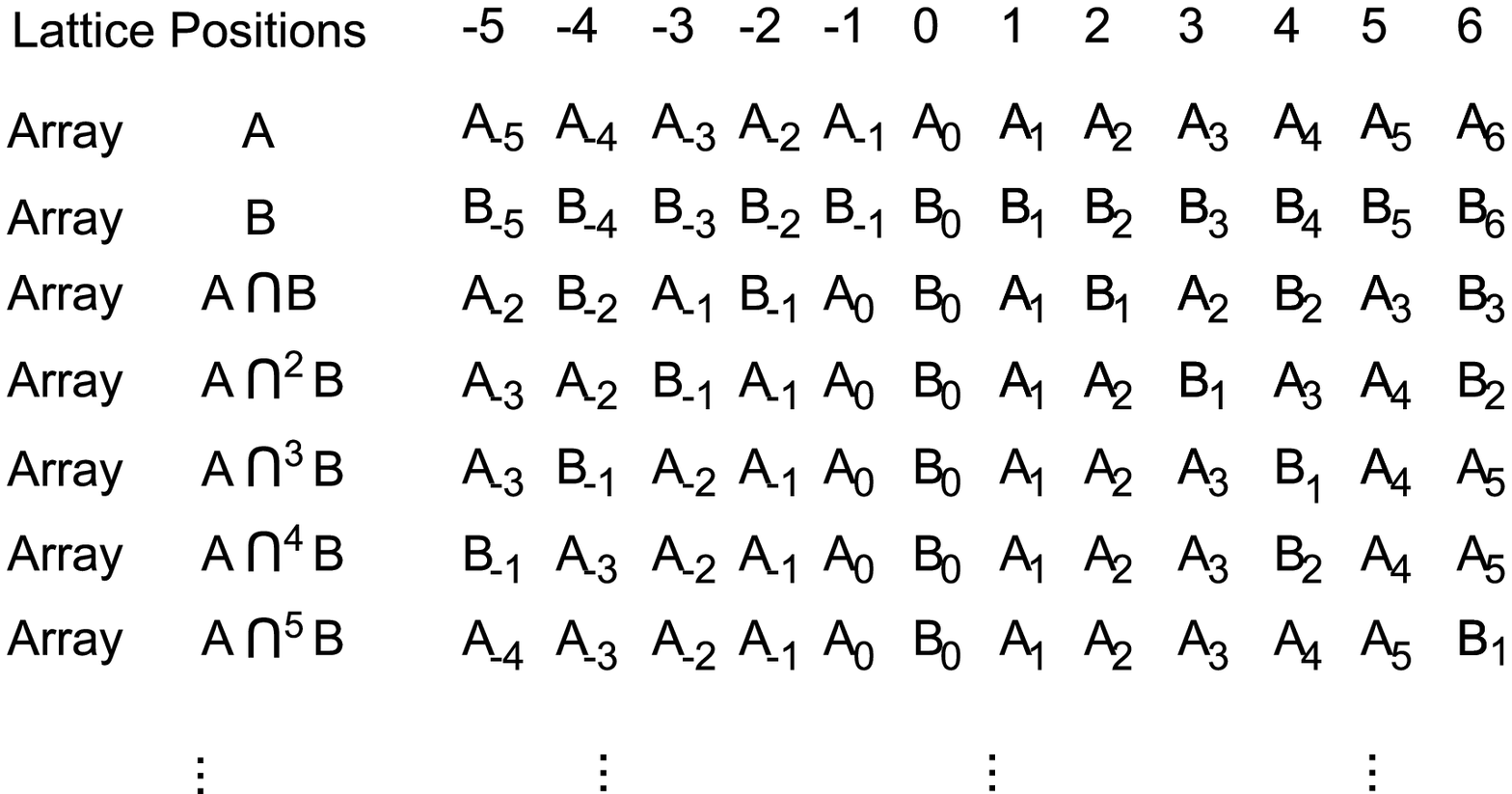}
\caption{The arrays $A$, $B$, and $A \cap^{i} B$ for $1 \leq i \leq 5$}
\label{capnexamples}
\end{figure}

\begin{defn}
Define $f^{n}_{X}(A)$ as the function sending an array $A$ to an array $A \cap^{n} X$.
\end{defn}

\begin{remark}
Note that the $\cap^{1}$ operation is equivalent to the $\cap$ operation defined previously.
\end{remark}

\subsection{$(\cap A \cap B)^{\infty}$}

We will now consider a subset of infinite compositions of the $\cap$ and $\cap^{n}$ operations. Note that the $\cap$ operation is equivalent to the operation $\cap^{1}$. Thus, in the $1$ dimensional case, $\cap^{n}$ can be considered to be a more general operation.
We will start with the basic definitions, before giving illuminating examples.

\begin{defn}[$(\cap^{s} A \cap^{t} B)^{m}$]
Consider two sequences $A$ and $B$. Let $s,t,m \in \mathbb{N}$. Then define $(\cap^{s} A \cap^{t} B)^{m}$ as follows;

\[(\cap^{s} A \cap^{t} B)^{m} := \underbrace{A \cap^{t} B \cap^{s} A \cap^{t} B \cap^{s} \ldots \ldots \cap^{s} A \cap^{t} B }_{2m}\]

Similarly (for sequences $A_{i}$, and integers $a_{i} \in \mathbb{N}$).

\[(\cap^{a_{1}} A_{1} \cap^{a_{2}} A_{2} \ldots \cap^{a_{n}} A_{n})^{m} := \underbrace{A_{1} \cap^{a_{2}} A_{2} \ldots \cap^{a_{n}} A_{n} \cap^{a_{1}} A_{1} \cap^{a_{2}} A_{2} \ldots \ldots \cap^{a_{n-1}} A_{n-1} \cap^{a_{n}} A_{n} }_{mn}\]
\end{defn}

\begin{defn}[$(\cap^{s} A \cap^{t} B)^{\infty}$]
Let $A$ and $B$ be sequences. Let $s,t \in \mathbb{N}$. Then define;

\[  (\cap^{s} A \cap^{t} B)^{\infty} := \lim_{m \rightarrow \infty} (\cap^{s} A \cap^{t} B)^{m} \] if the limit exists.

Similarly (for sequences $A_{i}$, and integers $a_{i} \in \mathbb{N}$),

\[  (\cap^{a_{1}} A_{1} \cap^{a_{2}} A_{2} \ldots \cap^{a_{n}} A_{n})^{\infty} := \lim_{m \rightarrow \infty} (\cap^{a_{1}} A_{1} \cap^{a_{2}} A_{2} \ldots \cap^{a_{n}} A_{n})^{m} \] if the limit exists.
\end{defn}

We will now work towards calculating the limit of repetitively applying a series of $\cap^{n} N$ operators, with different periodic arrays $N$.
We will eventually prove that you can create substitution sequences by this method.
To improve clarity, we will run through a simple example (with associated proofs) first, to give the reader some intuition into this area.

\subsection{ \( \underline{0} \cap \underline{1} \) and its limit \( (\underline{0} \cap \underline{1})^{\infty} \) }

Recall that $(\underline{0} \cap \underline{1})^{m}$ is defined as $\underbrace{\underline{0} \cap \underline{1} \cap \underline{0} \cap \underline{1} \ldots \cap \underline{0} \cap \underline{1}}_{2m}$.
Similarly, $(\underline{0} \cap \underline{1})^{\infty}$ is defined as the limit of $(\underline{0} \cap \underline{1})^{m}$ as $m \rightarrow \infty$.

We will show that $(\underline{0} \cap \underline{1})^{\infty}$ exists, and is equivalent to an (aperiodic primitive) substitution tiling.

Consider the sequence of arrays $ \ast, f_{\underline{1}}(\ast), f_{\underline{0}} \circ f_{\underline{1}} (\ast)$, $\ldots, (f_{\underline{0}} \circ f_{\underline{1}})^{m}(\ast), f_{\underline{1}} \circ ( f_{\underline{0}} \circ f_{\underline{1}})^{m}(\ast), \ldots$.
Part of this sequence is shown in figure \ref{tablepositionvsarraytypes}. Recall that $\ast$ is the null array.


\begin{figure}[!hbtp]
\includegraphics[angle=90, width=0.8\textwidth]{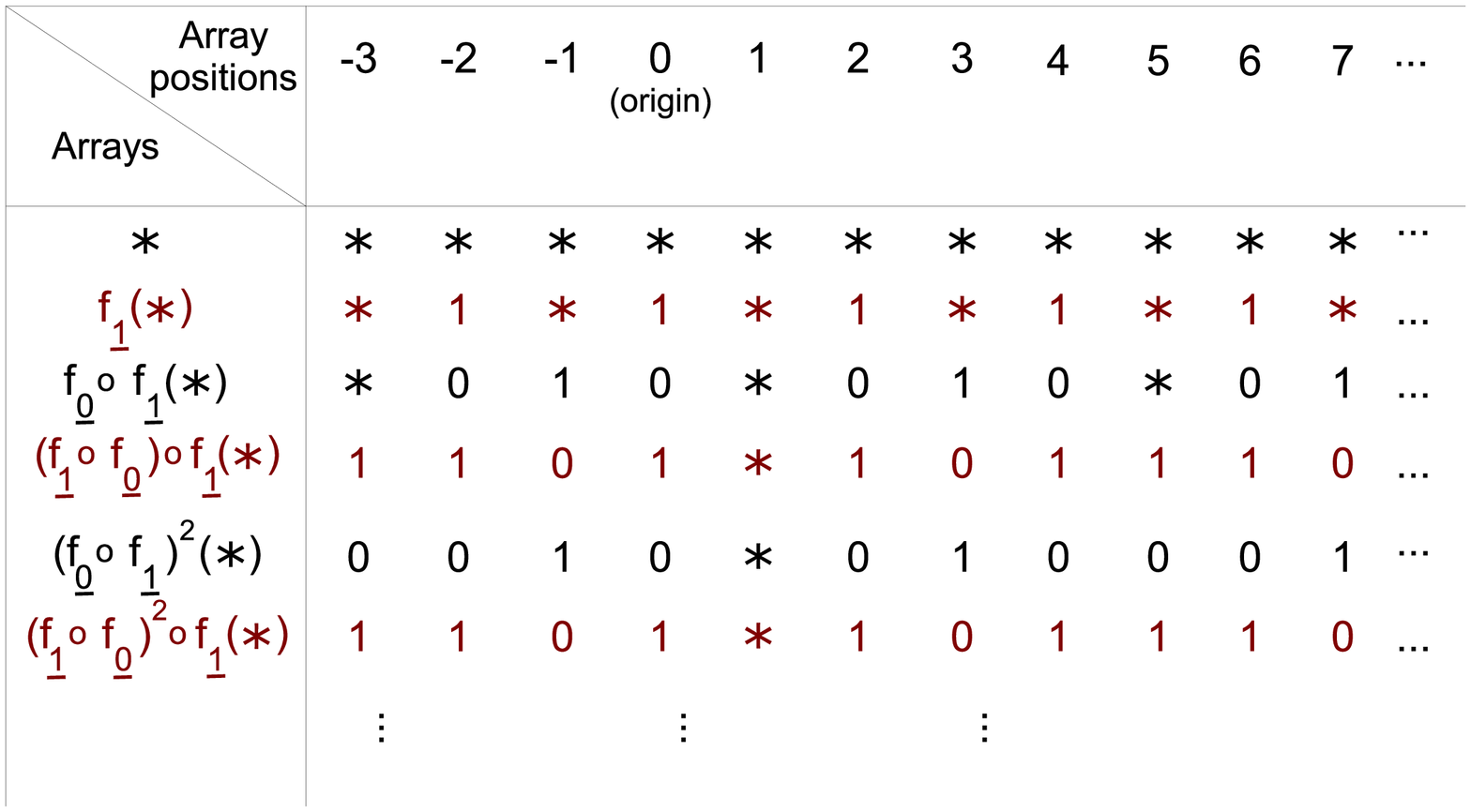}
\caption{Sequence of arrays.}
\label{tablepositionvsarraytypes}
\end{figure}

As we repetitively apply the $f_{\underline{0}} \circ f_{\underline{1}}$ operation to our array it appears that larger and larger patches of the array are independent of what the starting array was. We will now make this rigorous.

\begin{defn}
Consider an array $A_{1} \cap^{n_{2}} A_{2} \cap^{n_{3}} A_{3} \ldots \cap^{n_{m}} A_{m}$.
A point $p$ in the underlying lattice of $A_{1} \cap^{n_{2}} A_{2} \ldots \cap^{n_{m}} A_{m}$ is called  an \emph{undefined} point if $\ast \cap^{n_{2}} A_{2} \cap^{n_{3}} A_{3} \ldots \cap^{n_{m}} A_{m}$ maps  that point to $\ast$.
\end{defn}

Thus from definition \ref{Superpositionoperation}, any array of the form $X \cap \underline{1}$ (ie $f_{\underline{1}}(X)$), where $X$ is any array, has undefined points at odd positions.

\begin{lemma} \label{0cap1andffunction}
Let $m \in \mathbb{N}_{0}$ and $k \in \mathbb{N}$.
If a point is defined in $f_{\underline{1}} \circ (f_{\underline{0}} \circ f_{\underline{1}})^{m} (\ast)$, it will be defined, and take the same value, in $f_{\underline{1}} \circ (f_{\underline{0}} \circ f_{\underline{1}})^{m+k} (\ast)$.
\end{lemma}

\begin{proof}

Consider $f_{\underline{0}}(X)$, for any array $X$.
By the definition of the $\cap$ operation, we know that $f_{\underline{0}}(X)$ will have defined points at positions $2s$ ($s \in \mathbb{Z}$), with value $0$.
(Similarly, $f_{\underline{1}}(X)$ will have defined points with value $1$ in the same positions).

Consider $f_{\underline{1}} \circ f_{\underline{0}} (X)$ which is equivalent to $f_{\underline{1}}$ applied to $f_{\underline{0}}(X)$. By the definition of the $\cap$ operation, a point at position $4s-1$ in $f_{\underline{1}} \circ f_{\underline{0}} (X)$ will share the same value as a point at position $2s$ in $f_{\underline{0}}(X)$.
Thus $f_{\underline{1}} \circ f_{\underline{0}} (X)$ will have defined points at positions $4s-1$ with value $0$, and defined points at positions $2s$ with value $1$.

Let us generalise to an array $f_{i_{1}} \circ f_{i_{2}} \circ \ldots \circ f_{i_{n}} (X)$. This array will have defined points at positions $2^{n}s - 2^{n-1} +1$, with value $i_{n}$. We will use induction on $n$ to prove this.

Regarding the initial case, $f_{i_{1}}(X)$ has defined points at positions $2s$ as required (via definition $\ref{Superpositionoperation}$).
For the induction step, consider $f_{i_{2}} \circ f_{i_{3}} \circ \ldots \circ f_{i_{n+1}} (X)$. This will have defined points at positions $2^{n}s - 2^{n-1} +1$, with value $i_{n+1}$, from our assumption. Call this array $X_{n}$.
From definition \ref{Superpositionoperation}, $f_{i_{1}}  (X_{n})$ will have a point with value $X_{n}(v)$ at position $2v-1$.
Thus the set of points with value $i_{n+1}$ in $f_{i_{1}} (X_{n})$  are $\{x | x \in 2(2^{n}s -2^{n-1} +1) -1, s \in \mathbb{Z} \}$.
This can be rewritten as $S_{n+1}=\{x | x \in 2^{n+1}s -2^{n} +1, s \in \mathbb{Z} \}$, as required.
Thus our induction holds.

Thus $f_{\underline{1}} \circ (f_{\underline{0}} \circ f_{\underline{1}})^{m+k} (\ast)$ will define every point $f_{\underline{1}} \circ (f_{\underline{0}} \circ f_{\underline{1}})^{m} (\ast)$ defines.
This is because $f_{\underline{1}} \circ (f_{\underline{0}} \circ f_{\underline{1}})^{m} (\ast)$  will have defined points (with associated values) in $\bigcup_{i=1}^{i=2m+1} S_{i}$ and $f_{\underline{1}} \circ (f_{\underline{0}} \circ f_{\underline{1}})^{m+k} (\ast)$ will have defined points in $\bigcup_{i=1}^{i=2(m+k)+1} S_{i}$.
Furthermore, the defined points will have the same value, since the sets $S_{n} = \{2^{n}s - 2^{n-1} +1 | s \in \mathbb{Z}\}$ are disjoint for $n \geq 1$, thus the value of a point is solely determined by which set $S_{i}$ it is contained within.

For an explicit proof that the sets $S_{n}$ are disjoint, consider $S_{a}$ and $S_{b}$, for $a,b \in \mathbb{N}$, $a \neq b$. Assume WLOG that $a<b$. Then $b=a+c$, for $c>0$.
The sets $S_{a}$ and $S_{b}$ will have non-zero intersection if there exists an integer solution to the equation $2^{a}s - 2^{a-1} +1 = 2^{b}t - 2^{b-1} +1$.
This can be rewritten as $2^{a}s - 2^{a-1} +1 = 2^{a+c}t - 2^{a+c-1} +1$, and simplified to $2s - 1 = 2^{c+1}t - 2^{c}$. The left hand side of this equation can only take odd values, and the right hand side can only take even values (since $c>0$). Thus $S_{a}$ and $S_{b}$ are disjoint.

\end{proof}

\begin{thm} \label{limit0cap1exists}
$\lim_{m \rightarrow \infty} (\underline{0} \cap \underline{1})^{m}$ exists.
\end{thm}

\begin{proof}
Note that another equivalent way of writing $\lim_{m \rightarrow \infty} (\underline{0} \cap \underline{1})^{m}$ is as $\lim_{m \rightarrow \infty} [ f_{\underline{1}} \circ (f_{\underline{0}} \circ f_{\underline{1}})^{m-1}(\underline{0})]$.

Consider the lattice point at position \lq 1' in any array. For any array $X$, the $f_{X}$ operation does not change what value is assigned to this point.
An array $(\underline{0} \cap \underline{1})^{m}$ is constructed by taking a $\underline{0}$ array, and applying multiple $f_{\underline{1}}$ and $f_{\underline{0}}$ operations to it in turn. Thus the lattice point at position \lq 1' in an array $(\underline{0} \cap \underline{1})^{m}$ will be a point of type $0$, for all $m$. Thus the lattice point at position \lq 1' is well-defined in the limit.

For our next step, we will show that in the limit, the only undefined point is at position \lq 1'. Call this point the \emph{seed}.
Consider $f_{\underline{1}}(\ast)$. The points at positions $2$ and $0$ are defined.
On applying $f_{\underline{0}}$ to this array, the point at position $2$ is shifted to position $3$, and the point at position $0$ is shifted to position $-1$. Thus in $f_{\underline{0}} \circ f_{\underline{1}} (\ast)$, the points two units away from the seed point are defined from $f_{\underline{1}}$, and the points one unit away from the seed are now defined by $f_{\underline{0}}$.
Every time a new $f_{i}$ function is applied to the array, the patch of defined points grows by (at least) one.
Thus in the limit, every point (barring the seed point) is defined for some $f_{\underline{1}} \circ (f_{\underline{0}} \circ f_{\underline{1}})^{m} (\underline{0})$.
From the previous lemma, we know that every defined point will keep the same value as $m \rightarrow \infty$. Thus these lattice points are well-defined in the limit.
Thus $\lim_{m \rightarrow \infty} (\underline{0} \cap \underline{1})^{m}$ exists.

\end{proof}

\begin{thm} \label{0cap1istoeplitz}
$( \underline{0} \cap \underline{1})^{\infty}$ is an almost Toeplitz sequence, with the only undefined point being at position \lq 1'.
\end{thm}

\begin{proof}
We will use induction to show that every defined point in $( \underline{0} \cap \underline{1})^{\infty}$ is in some periodic part.
Construct the sequence $ \{X_{i}\}$ where $X_{1}=f_{\underline{1}}(\underline{0})$, $X_{2} = f_{\underline{0}}(X_{1})$, $X_{2i} = f_{\underline{0}}(X_{2i-1})$ and $X_{2i+1} = f_{\underline{1}}(X_{2i})$   for $i \geq 2$.

Consider the defined points of $X_{1}=f_{\underline{1}}(\underline{0})$. These points occur at even positions, and thus form a $2$-periodic part, as defined in definition \ref{periodicpart}. Thus every defined point in $X_{1}$ is in some periodic part. (In fact, every point is in a periodic part).

Assume every defined point in $X_{n}$ is in some periodic part. Consider $X_{n+1} = f_{\underline{i}}(X_{n})$ (where $f_{\underline{i}}$ is either $f_{\underline{0}}$ or $f_{\underline{1}}$, based on parity of $n$). (Denote the points of $X_{n}$ as $x'_{t}$, and the points of $X_{n+1}$ as $x_{t}$, for $t \in \mathbb{Z}$).

Any defined point in $f_{\underline{i}}(X_{n})$ will either correspond to a defined point in $X_{n}$, or will be an even point.
If it is an even point, it will be assigned a value of $1$ (if $n$ is even) or $0$ (if $n$ is odd).
Since all even points are assigned the same value, the even point of $f_{\underline{i}}(X_{n})$ form a $2$-periodic part.

Any defined point $x_{2t-1} \in f_{\underline{i}}(X_{n})$ will correspond to a point $x'_{t}$ in $X_{n}$.
By our assumption, the point $x'_{t}$ in $X_{n}$ belongs to some $p$-periodic part, for some $p$.
This periodic part consists of a series of points $x'_{t+kp}$ with the same label (for $k \in \mathbb{Z}$).
Therefore the $p$-periodic part will be mapped to a $2p$-periodic part in $f_{\underline{i}}(X_{n})$, consisting of a series of points $x_{2(t+kp)-1}$ with the same label.
Thus any defined point in $X_{n+1}$ belongs to some periodic part.
By induction, any defined point in $X_{i}$, for any $i$, must be in a periodic part.
We know from the last theorem that if a point is defined in $X_{i}$, it will have the same value in all sequences $X_{i+2k}$. Therefore once a point is defined (and thus belongs to a periodic part), it will always belong to a periodic part. We can therefore conclude that every defined point in the limit $( \underline{0} \cap \underline{1})^{\infty}$ is in a periodic part.





\end{proof}

\begin{remark}
Any finite composition of periodic arrays under the $\cap$ operation will be periodic. In particular, the sequence $(\underline{0} \cap \underline{1})^{m}$ is periodic, with period $2^{2m-1}$.

As a sketch proof, let $\{X_{k}\}_{k=0}^{\infty}$ be any arrays with only one letter in each alphabet, possibly different for each array. (A similar proof applies for periodic arrays, but requires more complex notation for little added clarity).
Consider some sequence $F_{1} = f_{X_{1}}(X_{0})$, $F_{i} = f_{X_{i}}(F_{i-1})$.
Consider a point in $X_{0}$. It will be in a $1$-periodic part.
Under the $f_{X_{1}}$ operation, this $1$-periodic part will be mapped to a $2$-periodic part. Any point in $F_{1}$ will either be undefined (and hence in the $2$-periodic part inherited from $X_{0}$), or it will be defined.
If it is defined, it will be in another $2$-periodic part, this one consisting of points from $X_{1}$ in even positions in the array.
Thus $F_{1}$ has period $2$.

Applying the $f_{X_{2}}$ operation to $F_{1}$ will send the $2$-periodic part representing undefined points in $F_{1}$ to a $4$-periodic part in $F_{2}$. The $2$-periodic part representing points first defined by $f_{X_{1}}$ will be sent to a $4$-periodic part, and a new $2$-periodic part will be created for the points from $X_{2}$ which have been placed in even positions in the array $F_{2}$.
Thus $F_{2}$ will have period $4$ (the \emph{lcm} of all the periodic parts).

Similarly $F_{3}$ will consist of two $8$-periodic parts, a $4$-periodic part and a new $2$-periodic part containing values from $X_{3}$. In general, $F_{i}$ will contain two $2^{i}$-periodic parts, and one $2^{j}$-periodic part for each $0 < j < i$. Thus $F_{i}$ will have period $2^{i}$.
$(\underline{0} \cap \underline{1})^{m}$ can be expressed as $F_{2m-1}$ for $X_{2k} = \underline{0}$, $X_{2k+1} = \underline{1}$.
Thus the result follows.


\end{remark}

\begin{thm}
The sequence $(\underline{0} \cap \underline{1})^{\infty}$ is a fixed point of the substitution,

\begin{align}
\sigma: & 0 \mapsto 0101 \nonumber \\
& 1 \mapsto 1101 \nonumber
\end{align}

More precisely $\sigma((\underline{0} \cap \underline{1})^{\infty}) =  (\underline{0} \cap \underline{1})^{\infty}$.

Furthermore $\lim_{m \rightarrow \infty} ((\underline{0} \cap \underline{1})^{m} \cap \underline{0})$ is also a fixed point of a substitution, namely;

\begin{align}
\widehat{\sigma}: & 0 \mapsto 0010 \nonumber \\
& 1 \mapsto 1010 \nonumber
\end{align}

More precisely $\widehat{\sigma} (\lim_{m \rightarrow \infty} ((\underline{0} \cap \underline{1})^{m} \cap \underline{0})) = \lim_{m \rightarrow \infty} ((\underline{0} \cap \underline{1})^{m} \cap \underline{0})$.
\end{thm}

\begin{proof}

First consider $\lim ((\underline{0} \cap \underline{1})^{m} \cap \underline{0})$. Note that this limit exists, since it is equivalent to $\lim f_{\underline{0}}((\underline{0} \cap \underline{1})^{m})$, and $\lim ((\underline{0} \cap \underline{1})^{m})$ is well defined.

We know that $\lim_{m \rightarrow \infty} ((\underline{0} \cap \underline{1})^{m} \cap \underline{0})$ must be invariant under $f_{\underline{0}} \circ f_{\underline{1}}$, since $f_{\underline{0}} \circ f_{\underline{1}}$ is equivalent to increasing the value of $m$ by one.
From theorem  \ref{0cap1istoeplitz}, we know that the point at position \lq 1' of $\lim ((\underline{0} \cap \underline{1})^{m} \cap \underline{0})$ must be of value $0$, since neither $f_{\underline{0}}$ or $f_{\underline{1}}$ can alter the value of the point at position \lq 1 ', and it is of value \lq $0$' in the array $\underline{0}$ (the first sequence we are building the limit from). Denote the point at position \lq 1' the \emph{seed} point.

Examining the proof of theorem \ref{0cap1istoeplitz}, note that points defined by the first application of $f_{\underline{0}}$ must be in a $2$-periodic part, and points defined by the following application of $f_{\underline{1}}$ will be in a $(2 \times 2)$-periodic part.
Thus we can deduce that the highest periodic part in $f_{\underline{0}} \circ f_{\underline{1}} (\ast)$ is of period $4$, thus $f_{\underline{0}} \circ f_{\underline{1}} (\ast)$ is periodic, with period $4$. Furthermore the repeating period is the word $\ast$ $0$ $1$ $0$.

Now consider a generic sequence $t = \ldots t_{0} t_{1} t_{2} \ldots$, and what points are defined from it under the application of $f_{\underline{0}} \circ f_{\underline{1}}$.
The seed point, $t_{1}$ gets sent to the word $t_{1}$ $0$ $1$ $0$. Via the periodicity of $f_{\underline{0}} \circ f_{\underline{1}} (\ast)$ we know that the point $t_{2}$ is sent to the word $t_{2}$ $0$ $1$ $0$, which is joined on to $t_{1}$'s word by concatenation. Similarly, $t_{3}$'s word is concatenated to the end of $t_{2}$'s word, and so on for all $t_{i}$. This is of course the function $\sigma$ from the definition of the substitution.
Since the seed point is fixed for all $m \geq 1$, the points generated from it are fixed for all $m \geq 2$. Similarly, all points generated from those points are fixed for $m \geq 3$, and so on. Thus all points defined via the application of $f_{\underline{0}} \circ f_{\underline{1}}$ are fixed in the limit.

Since every point in the underlying lattice is defined by $f_{\underline{0}} \circ f_{\underline{1}}(t)$, we can conclude that $\lim ((\underline{0} \cap \underline{1})^{m} \cap \underline{0})$ can be generated via a substitution.
Since $f_{\underline{0}} \circ f_{\underline{1}} (\lim ((\underline{0} \cap \underline{1})^{m} \cap \underline{0})) = \lim ((\underline{0} \cap \underline{1})^{m} \cap \underline{0})$, it is a fixed point of that substitution.


For $( \underline{0} \cap \underline{1})^{\infty}$ use a similar proof, using the fact that $( \underline{0} \cap \underline{1})^{\infty}$ is invariant under $f_{\underline{1}} \circ f_{\underline{0}}$.
\end{proof}


We will now generalize these proofs to arrays of the form $(\cap^{a_{1}} A_{1} \cap^{a_{2}} A_{2} \cap^{a_{3}} \ldots \cap^{a_{n}} A_{n})^{\infty}$, where $A_{i}$ are periodic arrays.

We will also show these types of arrays are (almost) Toeplitz arrays.

\subsection{$(\cap^{a_{1}} A_{1} \cap^{a_{2}} A_{2} \cap^{a_{3}} \ldots \cap^{a_{n}} A_{n})^{\infty}$}

We can extend the proofs of lemma \ref{0cap1andffunction} and theorem \ref{limit0cap1exists} to cover the more general case, $(\cap^{a_{1}} A_{1} \cap^{a_{2}} A_{2} \cap^{a_{3}} \ldots \cap^{a_{n}} A_{n})^{\infty}$. To avoid excessive duplication, we will sketch how to alter the proof to the more general case.

\begin{thm} \label{a1anlimitexists}
$(\cap^{a_{1}} A_{1} \cap^{a_{2}} A_{2} \cap^{a_{3}} \ldots \cap^{a_{n}} A_{n})^{\infty}$ is well defined.
\end{thm}

\begin{proof}[Sketch proof]
Take $\sigma = \text{ min}\{a_{1}, \ldots , a_{N} \}$. Any point in an array $X$ with position strictly between $0$ and $\sigma$ is invariant under any operation $\cap^{a_{i}} A_{i}$.
Thus the first $(\sigma - 1)$ points of $A_{1}$ form a \lq seed' which is invariant, and thus the values of the seed points are well defined in the limit.

Ignoring the seed, the first undefined point in $\cap^{a_{1}} A_{1} \cap^{a_{2}} A_{2} \cap^{a_{3}} \ldots \cap^{a_{n}} A_{n}$ cannot be closer to the origin that position $(\sigma + 1)$, since the first $(\sigma - 1)$ points are in the seed, and a new value will be inserted at position $\sigma$ during every iteration of $\cap^{a_{1}} A_{1} \cap^{a_{2}} A_{2} \cap^{a_{3}} \ldots \cap^{a_{n}} A_{n}$, by the definition of $\sigma$.

Similarly, the first undefined point of $(\cap^{a_{1}} A_{1} \cap^{a_{2}} A_{2} \cap^{a_{3}} \ldots \cap^{a_{n}} A_{n})^{2}$ cannot be closer to the origin than $\sigma + 2$, and in general the first undefined point of $(\cap^{a_{1}} A_{1} \cap^{a_{2}} A_{2} \cap^{a_{3}} \ldots \cap^{a_{n}} A_{n})^{m}$ cannot be closer than $(\sigma + m)$.

$(\cap^{a_{1}} A_{1} \cap^{a_{2}} A_{2} \cap^{a_{3}} \ldots \cap^{a_{n}} A_{n})^{m}$ can be rewritten as;

\[(f^{a_{n}}_{A_{n}} \circ f^{a_{n-1}}_{A_{n-1}} \circ  \ldots \circ f^{a_{1}}_{A_{1}})^{k} ((\cap^{a_{1}} A_{1} \ldots \cap^{a_{n}} A_{n})^{m-k} \]
for $m > k$.

Thus any defined point in $(\cap^{a_{1}} A_{1} \cap^{a_{2}} A_{2} \cap^{a_{3}} \ldots \cap^{a_{n}} A_{n})^{k}$ will have the same value for any $(\cap^{a_{1}} A_{1} \cap^{a_{2}} A_{2} \cap^{a_{3}} \ldots \cap^{a_{n}} A_{n})^{m}$ for all $m > k$.
Thus $(\cap^{a_{1}} A_{1} \cap^{a_{2}} A_{2} \cap^{a_{3}} \ldots \cap^{a_{n}} A_{n})^{\infty}$ is well defined.

\end{proof}

\begin{thm} \label{A1toAninftyisToeplitz}
If $A_{i}$ are periodic arrays, then $(\cap^{a_{1}} A_{1} \cap^{a_{2}} A_{2} \cap^{a_{3}} \ldots \cap^{a_{n}} A_{n})^{\infty}$ is an (almost) Toeplitz array.
\end{thm}

\begin{proof}
Note that if $A_{i}$ is a periodic array, every point in $A_{i}$ is in a periodic part.
The maximum period of these periodic parts is equal to the period of the array $A_{i}$, denoted $p(A_{i})$.

We wish to show that every defined point in $(\cap^{a_{1}} A_{1} \cap^{a_{2}} A_{2} \cap^{a_{3}} \ldots \cap^{a_{n}} A_{n})^{m}$ is in a periodic part (ignoring the seed).
Construct a sequence $C= \{c(i)\}_{i=1}^{\infty}$ of arrays where the $s$th term in the sequence consists of the array $A_{1}$ with $s$ $f^{a_{i}}_{A_{i}}$ operations applied to it in ascending cyclic sequence.

The first few terms of this sequence are;

$c(1)= A_{1} \cap^{a_{2}} A_{2}$.

$c(2)= A_{1} \cap^{a_{2}} A_{2} \cap^{a_{3}} A_{3}$

$c(3)= A_{1} \cap^{a_{2}} A_{2} \cap^{a_{3}} A_{3} \cap^{a_{4}} A_{4}$

$\ldots$

$c(n-1)= A_{1} \cap^{a_{2}} \ldots \cap^{a_{n}} A_{n}$

$c(n)= A_{1} \cap^{a_{2}} \ldots \cap^{a_{n}} A_{n} \cap^{a_{1}} A_{1}$

$c(n+1)= A_{1} \cap^{a_{2}} \ldots \cap^{a_{n}} A_{n} \cap^{a_{1}} A_{1} \cap^{a_{2}} A_{2}$

$\ldots$

$\ldots$.

Explicitly $c(s)$ is;

\[f^{a_{j}}_{A_{j}} \circ \ldots \circ f^{a_{1}}_{A_{1}} \circ f^{a_{n}}_{A_{n}} \circ f^{a_{n-1}}_{A_{n-1}} \circ \ldots \circ f^{a_{1}}_{A_{1}} \circ \ldots \ldots \circ f^{a_{2}}_{A_{2}} (A_{1})\]
$\ldots$where $j$ is such that $s=(n-1)+(k-1)n+j$, and $j \leq n$.

The arrays $(\cap^{a_{1}} A_{1} \cap^{a_{2}} A_{2} \cap^{a_{3}} \ldots \cap^{a_{n}} A_{n})^{m}$ form a subsequence of this sequence.
Thus if we can show that every defined point in any array (from $C$) is in a periodic part, ignoring the seed, then we are done.
We will use induction, showing that this statement holds true for $c(1)$, and that if the statement is true for $c(s)$ it is true for $c(s+1)$.

For the initial step of our induction, consider $f^{a_{2}}_{A_{2}} (\ast)$, alternatively known as $\ast \cap^{a_{2}} A_{2}$.
The points defined by $f^{a_{2}}_{A_{2}}$ belong to a finite number of periodic parts, with maximum period $(a_{2}+1).p(A_{2})$. Thus every defined point in $f^{a_{2}}_{A_{2}} (A_{1})$ belongs to a periodic part.



For our induction step, assume every defined point in the array $c(s)$ is in a periodic part.
Without loss of generality, take any one of these periodic parts, $p$, of period $x$.
Consider $f^{a_{j+1}}_{A_{j+1}}$.
Consider a section of $c(s)$ of length $\text{lcm}(x,a_{j+1})$.
The operation $f^{a_{j+1}}_{A_{j+1}}$ adds a point every $a_{j+1}$ units. Thus under the $f^{a_{j+1}}_{A_{j+1}}$ operation any section $sec$ of length $\text{lcm}(x,a_{j+1})$ is expanded to a section $\grave{sec}$ of length
\[l = \text{lcm} (x,a_{j+1}) + \frac{\text{lcm}(x,a_{j+1})}{a_{j+1}}.\]

Furthermore, after the section of length $\text{lcm}(x,a_{j+1})$ has been expanded, the following section in $c(s)$ of length $\text{lcm}(x,a_{j+1})$ will also be expanded to a section of length $l$. Via the properties of the $\text{lcm}$ function, the positions of points from $p$ will be the same in $\grave{sec}$, the following section of length $l$, and all following sections after that.
Thus a $x$-periodic part in $c(s)$ will be turned into a $l$-periodic part in $f^{a_{n+1}}_{A_{n+1}}(c(s))$.

Thus by induction, every defined point (not in the seed) in $(\cap^{a_{1}} A_{1} \cap^{a_{2}} A_{2} \cap^{a_{3}} \ldots \cap^{a_{n}} A_{n})^{\infty}$ is in a periodic part.
By theorem \ref{periodicparttoeplitz}, $(\cap^{a_{1}} A_{1} \cap^{a_{2}} A_{2} \cap^{a_{3}} \ldots \cap^{a_{n}} A_{n})^{\infty}$  is (almost) Toeplitz.
\end{proof}

We will now introduce a natural extension of the well-known concept of a substitution.
The major change will be that the defining function of the substitution is based on \emph{words}, not letters.

\begin{defn}
Let $A$ be an alphabet, $A^{\ast}$ be the set of finite words over $A$, and $W \subset A^{\ast}$.
Let $A^{\infty}$ be the set of bi-infinite words, and $A^{\infty}w$ be the set of bi-infinite words with unique decomposition into a sequence of elements from $W$.
Let $A^{\ast}w$ be the set of finite words with unique decomposition into a sequence of elements from $W$.

The function  $\sigma : W \mapsto A^{\ast}$ is a valid substitution on words if it satisfies the following two properties;
Firstly it must induce a well defined function  $\sigma^{\infty} : A^{\ast}w \mapsto A^{\ast}w$.
Secondly for all $w \in W$, there exists $N$ such that $|w| < | \sigma^{n} (w)|$.

A $\sigma$-substitution tiling is an element $\Lambda \in A^{\infty}w$ when $\Lambda \subset Im (\sigma^{\infty})^{N}$ for all $N \in \mathbb{N}$.


For example a function $\sigma$ defined as $\sigma (01) = 0110$ , $\sigma (10) = 1001$  would induce a function $\sigma^{\infty}$ sending $ 011010 \ldots$ to $ 011010011001 \ldots$.

\end{defn}

\begin{thm}

$(\cap^{a_{1}} A_{1} \cap^{a_{2}} A_{2} \cap^{a_{3}} \ldots \cap^{a_{n}} A_{n})^{\infty}$ is a fixed point of a substitution on words.
\end{thm}

\begin{proof}

For brevity, we will define the function $f^{a_{n}}_{A_{n}} \circ f^{a_{n-1}}_{A_{n-1}} \circ \ldots \circ f^{a_{1}}_{A_{1}}$ as $F$.
Consider $F(\ast)$.
By the proof of theorem \ref{A1toAninftyisToeplitz}, points defined by $f^{a_{i}}_{A_{i}}$ will be in a $p_{i}$-periodic part, where $p_{i}$ is some finite number. Thus, since we have only applied a finite number of $f^{a_{i}}_{A_{i}}$ operations, there are a finite number of periodic parts. Thus $F(\ast)$ is periodic (with period less than or equal to $p_{1}\cdot p_{2} \cdot \ldots \cdot p_{n}$).

Since $F(\ast)$ is periodic, there must be a repeating sequence $S$ of minimum length (starting at the origin). For example, let $K_{i}$ be an array with alphabet $\{ k_{i} \}$. Then $f_{K_{2}} \circ f^{2}_{K_{1}} (\ast)$ would have a repeating sequence $\ast$ $k_{2}$ $\ast$ $k_{2}$ $k_{1}$ $k_{2}$.

By theorem \ref{a1anlimitexists}, there exists a finite number of \lq seed points' near the origin that are unchanged by $f^{a_{i}}_{A_{i}}$, for any $i \in \mathbb{N}$.
Call the word formed by these points $w = w_{1} w_{2} \ldots w_{k}$.
Consider the number $n$ of undefined points in $S$. If $|w| \geq n$, then we can repetitively apply our function $F$ to the first $n$ points of $w$ to define the entire tiling.
This is equivalent to a substitution rule sending $w_{1} \ldots w_{n}$ to the first $|S|$ terms of $F(\grave{w})$ where $\grave{w} = w_{1} w_{2} \ldots w_{n} \ast \ast \ast \ldots$.
For our earlier example of $f_{K_{2}} \circ f^{2}_{K_{1}} (\ast)$, this would give a substitution rule $w_{1} w_{2} \mapsto w_{1}$ $k_{2}$ $w_{2}$ $k_{2}$ $k_{1}$ $k_{2}$.

If $|w| < n$ then repetitively apply $F$ to $\grave{w}$ until the first $n$ points are defined. Then apply $F$ to these $n$ points as before.

\end{proof}

Let us consider some examples, in order to improve our intuition regarding this process.

\subsection{Examples}

Define $\underline{X}$ as the array where all points in the array have type $X$, where $X$ can be any fixed symbol.

\begin{example}[$(\cap \underline{0} \cap \underline{1} \cap \underline{2})^{\infty}$]

\begin{IEEEeqnarray*}{rCl}
(f_{\underline{2}} \circ f_{\underline{1}} \circ f_{\underline{0}}) (\ast) & = &  f_{\underline{2}} \circ f_{\underline{1}} ( \ast 0 \ast 0 \ast 0 \ldots)
\\
& = & f_{\underline{2}} ( \ast 1 0 1 \ast 1 0 1 \ldots )
\\
& = & \ast 2 1 2 0 2 1 2 \ast 2 1 2 0 2 1 2 \ldots
\end{IEEEeqnarray*}

The repeating sequence is thus $\ast 2 1 2 0 2 1 2$.
Thus the substitution is $x_{i} \mapsto x_{i} 2 1 2 0 2 1 2$.

The first array in $(\cap \underline{0} \cap \underline{1} \cap \underline{2})^{\infty}$ is the $\underline{0}$ array, and only the point at the origin in $\underline{0}$ is unaffected by $f_{\underline{0}}$, $f_{\underline{1}}$ and $f_{\underline{2}}$.
Thus the seed of $(\cap \underline{0} \cap \underline{1} \cap \underline{2})^{\infty}$ is the word \lq $\text{ }0$'.
\lq $\text{ }0$' has one point in it, which is equal to the number of unknown points in $\ast 2 1 2 0 2 1 2$.
Thus $0$ is our initial word, and $x_{i} \mapsto x_{i} 2 1 2 0 2 1 2$ is our substitution.

\end{example}

\begin{example}[$(\cap^{2} \underline{A} \cap \underline{B} \cap^{4} \underline{C})^{\infty}$]

Define a sequence as follows: $\hat{\ast} = \ast_{1} \ast_{2} \ldots \ast_{n} \ast_{n+1} \ldots$
Now consider $f^{4}_{\underline{C}} \circ f_{\underline{B}} \circ f^{2}_{\underline{A}} (\hat{\ast})$.

\begin{IEEEeqnarray*}{rCl}
f^{4}_{\underline{C}} \circ f_{\underline{B}} \circ f^{2}_{\underline{A}} (\hat{\ast}) & = &  f^{4}_{\underline{C}} \circ f_{\underline{B}} ( \ast_{1} \ast_{2} A \ast_{1} \ast_{2} A \ldots)
\\
& = & f^{4}_{\underline{C}} ( \ast_{1} B \ast_{2} B A B \ast_{1} B \ast_{2} B A B \ldots )
\\
& = & \ast_{1} B \ast_{2} B C A B \ast_{1} B C \ast_{2} B A B C \ldots
\end{IEEEeqnarray*}

Therefore our substitution is:

\[x_{2n-1} x_{2n} \mapsto x_{2n-1} B x_{2n} B C A B x_{2n-1} B C x_{2n} B A B C  \]

This substitution rule has two unknown points ($x_{2n-1}$ and $x_{2n}$), but the seed of $(\cap^{2} \underline{A} \cap \underline{B} \cap^{4} \underline{C})^{\infty}$ has only one point in it (namely \lq $A$').
Thus apply $(f^{4}_{\underline{C}} \circ f_{\underline{B}} \circ f^{2}_{\underline{C}}$ to the seed until the first $2$ points are defined, as follows.

\begin{IEEEeqnarray*}{rCl}
f^{4}_{\underline{C}} \circ f_{\underline{B}} \circ f^{2}_{\underline{C}}(A \ast \ast \ast \ast \ldots) & = &  f^{4}_{C} \circ f_{\underline{B}} (A \ast A \ast \ast A \ast \ast A \ldots )
\\
& = & f^{4}_{\underline{C}} ( A B \ast B A B \ast B \ast \ldots )
\\
& = & A B \ast B C A B \ast B  \ldots
\end{IEEEeqnarray*}
Thus the seed of our substitution is $AB$.
\end{example}

\section{The $\cap$ operation in two dimensions.}

We have described the $\cap$ operation in one dimension. We will now extend the $\cap$ operation to higher dimensions and illustrate how this can be used to recreate an aperiodic tiling. For reasons of brevity we will merely give the definition of the operation in two dimensions, and a motivating example. Full details and proofs can be found in [REF THESIS].




As the reader may recall (ie figure \ref{layers1D}), in the $1$ dimensional case we used overlaying infinite strips as a physical motivation behind the $\cap$ operation.
Any generalization to two dimensions should be roughly equivalent to taking two unit square tilings on overhead transparencies, overlaying them, and reading off a new tiling from the result.

Consider the following definition.

\begin{defn}[$A \cap B$ for $\mathbb{Z}^{2}$-arrays]

Let $A: \mathbb{Z}^{d} \mapsto \Sigma$ and $B: \mathbb{Z}^{d} \mapsto \Sigma$ be $\mathbb{Z}^{2}$ arrays with the standard lattice with points at integer positions.

$A \cap B$ is an array with alphabet $\Sigma$. The underlying lattice of the array differs from the standard lattice, and not in an intuitive way.

Let $f_{1}=(\frac{1}{\sqrt{2}},\frac{1}{\sqrt{2}})$, $f_{2}=(\frac{-1}{\sqrt{2}},\frac{1}{\sqrt{2}})$.

Note that $|f_{1}|=|f_{2}|=1$.
Let $Latt_{A \cap B} = \{a_{1}f_{1} + a_{2}f_{2} | a_{1}, a_{2} \in \mathbb{Z}^{d}\}$ be the underlying lattice for $A \cap B$.

Then define;

\begin{equation*}
(A \cap B)(m,n) =
\begin{cases}
A((\frac{m-n}{2},\frac{m+n}{2})) & \text{if } m + n \text{ is even},\\
B((\frac{m-n-1}{2},\frac{m+n-1}{2})) & \text{if } m + n \text{ is odd}.

\end{cases}
\end{equation*}

\end{defn}

\begin{defn} \label{gAX}
Let $X$ and $A$ be $\mathbb{Z}^{2}$ arrays.
Then define $g_{A}(X)$ to be the function sending $X$ to $X \cap A$.
\end{defn}

These definitions are of the same format as the definition of $A \cap B$ in the one dimensional case.
It can be rephrased as follows

A point $x e_{1} + y e_{2}$ in $A$ is sent to the point $ x f_{1} - y f_{2}$ in $A \cap B$.
A point $x e_{1} + y e_{2}$ in $B$ is sent to the point $ (x+1) f_{1} + y f_{2}$ in $A \cap B$.
This covers the whole of the array $A \cap B$.

\begin{defn}
Denote an array with underlying lattice generated by the vectors $e_{1}, e_{2}$ as a \emph{(unit) square array}.
Denote an array with underlying lattice generated by the vectors $f_{1}, f_{2}$ as a \emph{diamond array}.
\end{defn}

For an intuitive construction of the $(A \cap B)$ operation, see the following remark;

\begin{remark}
Take two unit square tilings, $A$ and $B$ such that the origin of $\mathbb{R}^{2}$ is the centre of a tile in both $A$ and $B$. Switch to the lattice view of tilings (ie, consider the array with underlying lattice formed by the centres of the unit square tiles).
Shift $B$ by the vector $v=(\frac{\sqrt{2}}{2},\frac{\sqrt{2}}{2})$.
Overlap the two lattices. The vector $v$ has been chosen to maximize the distance between the set of tile centres of $A$, and the set of tile centres from $B$, i.e. maximizing $d(A,B)$.
From the tile viewpoint, this corresponds to choosing the vector $v$ to place the centre of tiles from $B$ at the vertices of tiles from $A$.

Scale up so that $d(B,A)=1$. The resultant array is no longer generated by the standard basis of $\mathbb{R}^{2}$. We will thus change the generating vectors, to a minimum vector between points in $A$ and $B$, and a vector orthogonal to that minimum vector.

This is reflected in the non-intuitive function \lq$\frac{n-m-1}{2}$', and the basis vectors $e_{1}$ and $e_{2}$ changing to $f_{1}$ and $f_{2}$.
See figure \ref{AandBandAcapB} for a pictorial representation of $A$, $B$ and $A \cap B$.
\end{remark}

\begin{figure}[!hbtp]
\includegraphics[angle=0, width=0.8\textwidth]{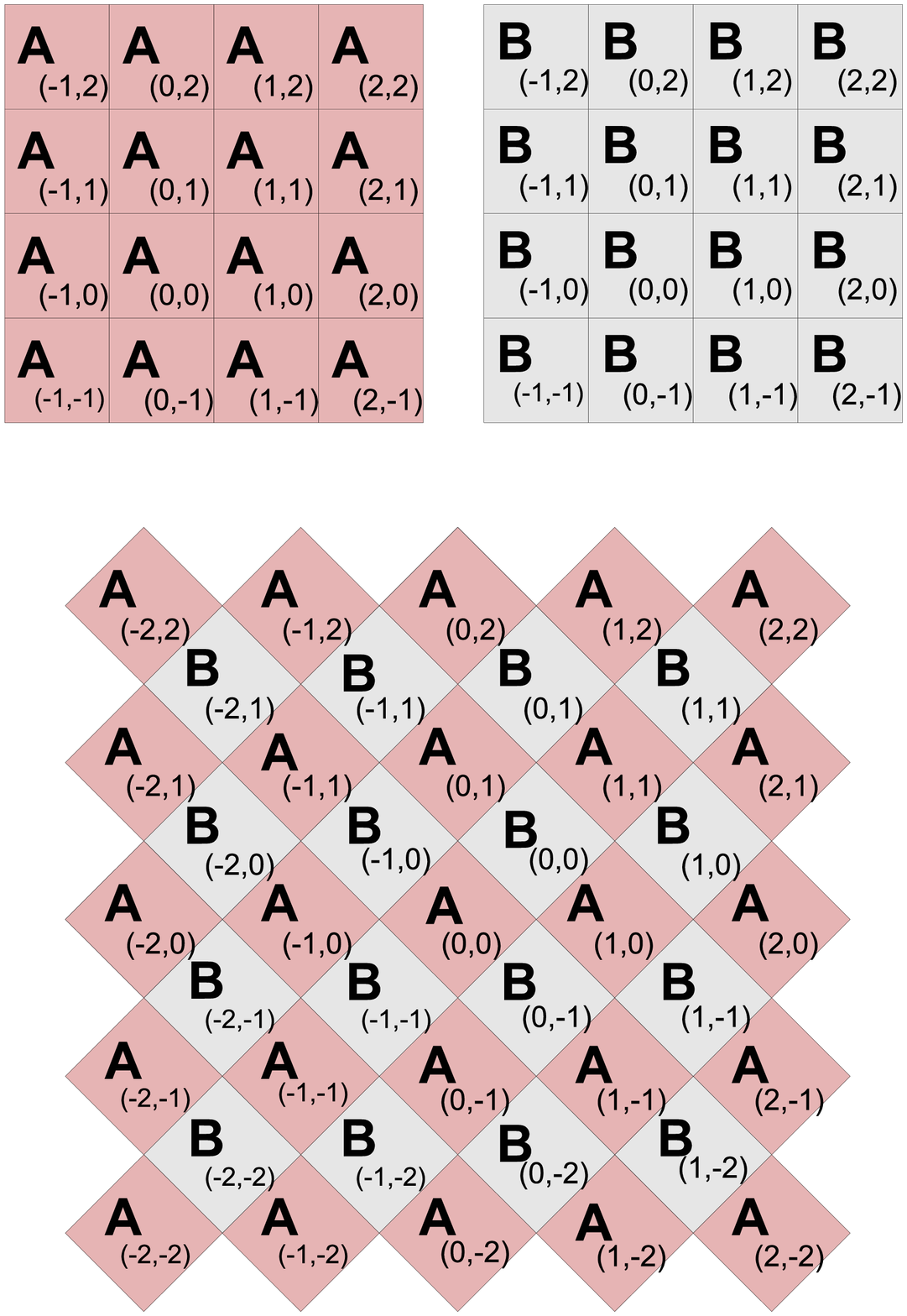}
\caption{Patches of the tilings $A$, $B$ and $A \cap B$}
\label{AandBandAcapB}
\end{figure}









As the reader will be aware, $A\cap B$ has a different underlying lattice to $A$ or $B$. The corresponding tiling is not a tiling by unit squares, but by unit diamonds.

If we wish to consider multiple $\cap$ operations, it is much more natural to consider sequential \emph{pairs} of $\cap$ operations.

\subsection{$\cap A \cap B$}

We already have a function $g_{A}$ (from definition \ref{gAX}) which applies an unit square array $A$ to some unit square array $X$ to produce a diamond array $X \cap A$.

In order to consider the limit of this $g_{A}$ function (and hence the $\cap$ function), we must find out how to apply it to a diamond tiling.
Let $A$ and $B$ be unit square arrays. Let X be some unit square array.

\begin{remark}
Like in the $1$D case, we will use a null array to help clarify how the $(\cap A \cap B)$ operation builds up the periodic parts of $(\cap A \cap B)^{\infty}$ through each application of $(\cap A \cap B)$, and to show how much of the array is uniquely defined by the repeated application of $(\cap A \cap B)$.
\end{remark}

In the previous section, we showed that $X \cap A$ is a diamond array. Thus $X \cap A$ has different basis vectors to a unit square array.
Let us see what effect naively applying the $\cap$ operation to a diamond array will have.

From the previous chapter, the $g_{B}$ operation applied to any array $X$, will double the distance of every point from the origin, (and then add points from $B$).

Our points in $(X \cap A)$ are of the form $\frac{1}{\sqrt{2}}(a-b,a+b)$ for $a,b \in \mathbb{Z}$.

Thus $g_{B} (X \cap A) = 2 (a-b, a+b)$ for $a,b \in \mathbb{Z}$, which, when represented as a tiling, is the subset shown in figure \ref{capBtoXcapA} (points from $X$ and $A$ are labelled as such).

\begin{figure}[!hbtp]
\includegraphics[angle=90, width=0.8\textwidth]{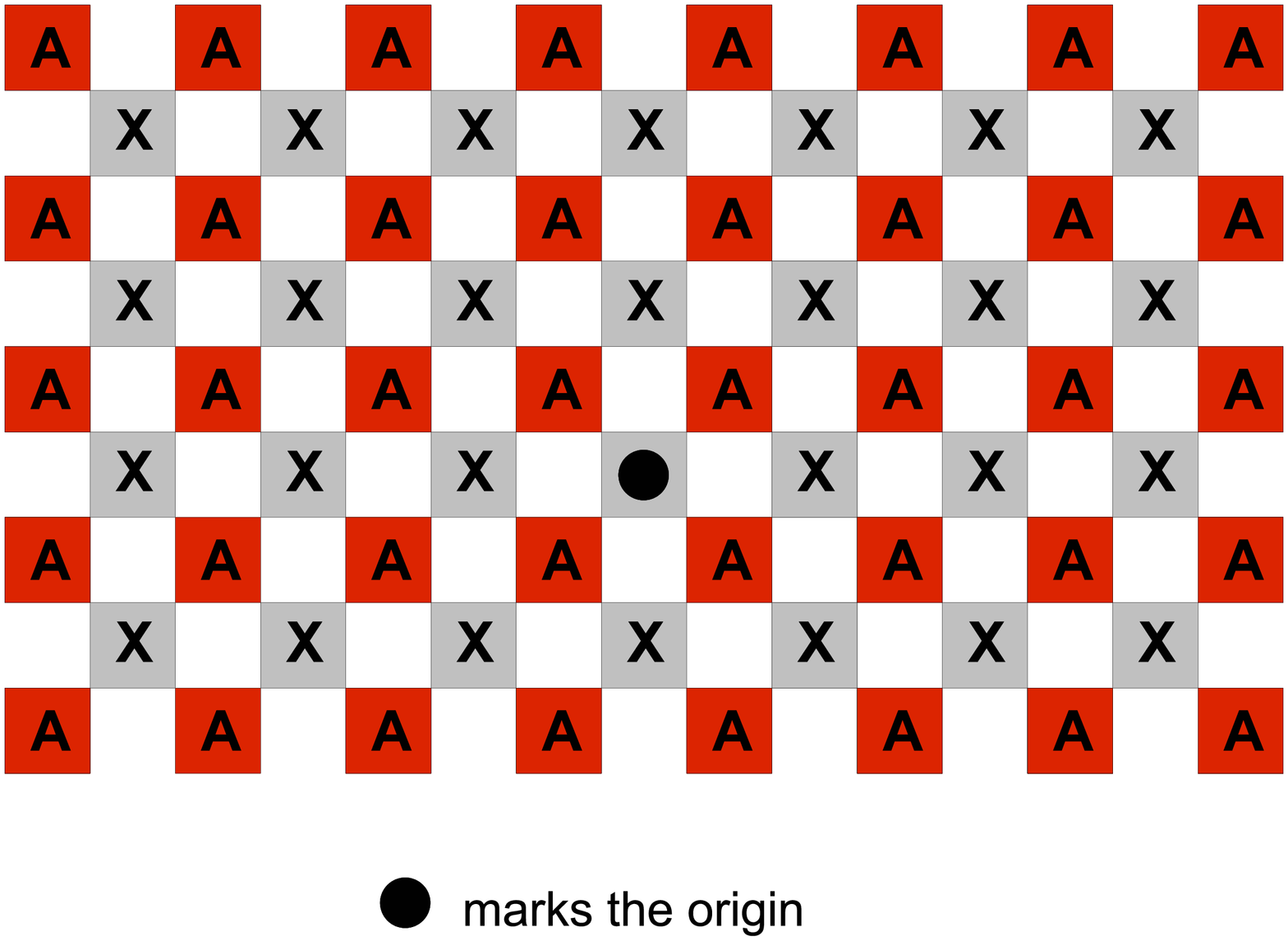}
\caption{The image of $X \cap A$ under the $\cap$ operation}
\label{capBtoXcapA}
\end{figure}

This is precisely half of an unit square array. Thus we need to ensure that the points from array $B$ fill the gaps.
If $B$ was a diamond array, this would simply be a case of scaling $B$ by $\sqrt{2}$, and shifting it by a small vector.
Of course, $B$ is not a diamond array (though you could easily define a related $\cap$ operation where $B$ is a diamond array.

Instead we will alter how the $\cap$ operation will effect points from the $B$ array, as follows;

\begin{defn}[$A \cap B$, if $A$ is a diamond array]
Let $A$ be a diamond array, with basis vectors $f_{1}, f_{2}$.
Let $B$ be a unit square array, with standard euclidean basis vectors $e_{1}, e_{2}$.

Then

\begin{equation*}
A \cap B(m,n) =
\begin{cases}
A((\frac{m+n}{2},\frac{-m+n}{2})) & \text{if } m + n \text{ is even},\\
B((\frac{m+n-1}{2},\frac{-m+n+1}{2})) & \text{if } m + n \text{ is odd}.

\end{cases}
\end{equation*}

The underlying lattice for $A \cap B$ will be $\{n_{1} e_{1} + n_{2} e_{2} | n_{1}, n_{2} \in \mathbb{Z}\}$.
\end{defn}

This definition can be rephrased as follows.

A point $x f_{1} + y f_{2}$ in $A$ is sent to the point $ (x-y) e_{1} + (x+y) e_{2}$ in $A \cap B$.
A point $x e_{1} + y e_{2}$ in $B$ is sent to the point $ (x+-y+1) e_{1} + (x+y) e_{2}$ in $A \cap B$.
This covers the whole of the array $A \cap B$.

Effectively, this function takes the underlying lattice of the square array $B$, rotates it anti-clockwise by $\frac{\pi}{4}$, then moves the point of the lattice which was over the origin, to the point $(0,1) \in \mathbb{Z}^{2}$.
The underlying lattice of $A$ is then used to fill in the gaps in $\mathbb{Z}^{2}$, as in figure \ref{BintoXcapA}.

\begin{figure}[!hbtp]
\includegraphics[angle=90, width=0.8\textwidth]{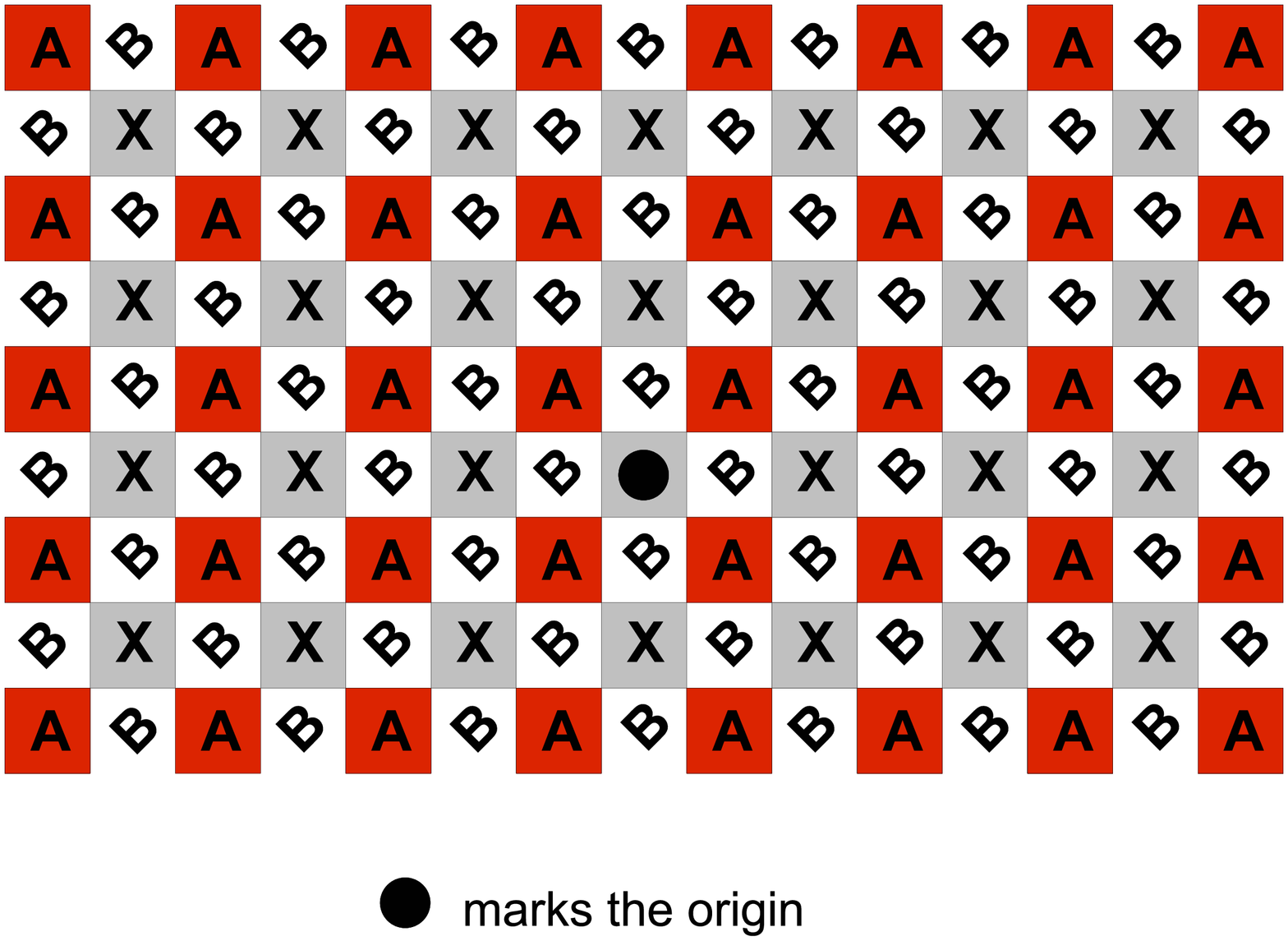}
\caption{$(X \cap A \cap B)$ (schematic)}
\label{BintoXcapA}
\end{figure}


Obviously the rotation effect is not ideal, but its effect can be ignored by choosing a $B$ tiling which has only one prototile (as we will do later).

\subsection{Motivation}

As motivation for the $\cap$ operation applied to diamond tilings, let us extrapolate from our physical motivation with square tilings.
With square tilings, we aimed to \lq overlay' the tilings, placing the second square tiling $B$ such that every centre of a tile in $B$ was as far away from the centre of tiles in $A$ as possible. We would then scale up this new tiling to ensure that the minimum distance between centres of tiles was $1$.

The problem we will get by applying this motivation to a diamond tiling ($X \cap A$) and a square tiling $B$ is that we have to rotate $B$ to ensure the points of $B$ are as far away from points of $A$ as possible (ie, maximising $d(B,A)$).


We will have several choices of how much we rotate $B$ by, namely ($\frac{\pi}{4} + k \frac{\pi}{2}$, $k \in \{0,3\}$), just as we have several choices of how far we translate $B$ by (any $(a,b)$ such that $a+b$ is odd).
We will choose $\frac{-\pi}{4}$ and $(0,1)$ in this thesis.

We have defined $X \cap A$, when $X$ is a diamond array and $A$ a square array. We will now define $\cap A \cap B$.

\begin{defn}
Let $X$ be a unit square tiling.
Let $A$ and $B$ be unit square tilings.
Then $X \cap A \cap B$ is evaluated from left to right, as in the $1$ dimensional case.
Similarly longer strings of $\cap$ operations are evaluated from left to right, for example;
$X(\cap A \cap B)^{n} =(X ( \cap A \cap B)^{n-1}) \cap A) \cap B$

Also, if the limit of $X(\cap A \cap B)^{n}$ as $n \rightarrow \infty$ exists, denote it as $X(\cap A \cap B)^{\infty}$.

\end{defn}


We can also consider the $g_{B}$ operation as a function from the domain of unit square tiling into the domain of diamond square tilings.
Similar functions can be defined from diamond square tilings into unit square tilings (distinguishable by the domain specified).

\begin{defn}
Define the function $g_{B} \circ g_{A} (X)$ as $f_{B} (X \cap A)$.
\end{defn}

\begin{remark}
Note that $\cap A \cap B$ doubles the distance of any point $(x,y) \in X$ from the origin.
This is because $\cap A$ sends the $(x,y)$ to $\sqrt{2} (x,y)$, then $\cap B$ sends $\sqrt{2}(x,y)$ to $2(x,y)$.
Thus in the limit of $(\cap A \cap B)^{n}$. there will be only one tile from $X$ remaining, at the origin. Thus in the current form, $X \cap (A \cap B)^{\infty}$ is not quite Toeplitz, since the origin is not defined solely by $\cap A \cap B$.
In later examples, we will set $X = B$ because of this problem, to sidestep it.
\end{remark}

\begin{thm}
If $A$ and $B$ are unit square arrays, $(\cap A \cap B)^{\infty}$ is well defined.
\end{thm}

\begin{proof}
Firstly, note that $X( \cap A \cap B)^{n}$ is a unit square tiling, for all $n$. Thus all we need to show is that each tile $t$ has an unique label, for all $X( \cap A \cap B)^{n}$ where $n > N_{t}$, where $N_{t} \in \mathbb{N}$.

Divide $X \cap A \cap B$ into three sets; the origin, points $(x,y)$ where $x+y$ is even, and tiles $(x+y)$ where $x+y$ is odd.
The origin of $X \cap A \cap B$ is unchanged under the operation $\cap A \cap B$ (being the origin of $X$).

The odd tiles of $X \cap A \cap B$ only depend on the tile types of $B$.
Since any tiling $X(\cap A \cap B)^{n}$ can be rewritten as $(X \cap A \cap B \ldots \cap B) \cap A \cap B$, these tiles are fixed for all $n \geq 1$.

Let us consider the tiles where $x+y$ is even.
The tiles in $X \cap A \cap B$ where $x$ \emph{and} $y$ are odd are uniquely defined by $A$.
Thus, since $X(\cap A \cap B)^{n}$ can be rewritten as $(X \cap A \cap B \ldots \cap B) \cap A \cap B$, this holds true for $X( \cap A \cap B)^{n}$ as well.
For the tiles where $x$ and $y$ are \emph{both even}, recall that the $g_{B} \circ g_{A}$ function doubles the distance of every point from the origin.
We know any point $(x,y)$ where $x+y$ is odd, is fixed for $n \geq 1$. Thus any point $(a,b) = 2^{k} (x,y)$ where $x+y$ is odd is fixed for $n \geq 1 +k$.
Call this the doubling property.

Take a point $(x,y)$, where $x$ and $y$ are even. By the prime factorization theorem, $x$ and $y$ can be decomposed uniquely into factors.
Thus $(x,y)$ can be rewritten as $(2^{q}s, 2^{r}t)$, where $s$ and $t$ are odd numbers (possibly $1$), and $q$ and $r$ are integers.
WLOG assume $q > r$. Via the doubling property, we know that after $r$ iterations of $g_{B} \circ g_{A}$, $(2^{q}s, 2^{r}t)$ will have the same label as $(2^{q-r}s,t)$, which is in the fixed set dependent on $B$ (since $2^{q-r}s$ is even, $t$ is odd, thus $2^{q-r}s +t$ is odd).
If $q=r$, then we have $(2^{q}s,2^{r}t)$ having the same label as $(s,t)$, where $s$ and $t$ are both odd numbers and thus fixed by $A$.

Thus every point in $X \cap A \cap B$ is fixed in the limit.
\end{proof}

For clarification, figure \ref{PartitionofXcapAcapBfinal} illustrates a patch of $X(\cap A \cap B)^{\infty}$, with colours added to indicate which tiles are defined by which iteration of $g_{A}$ and $g_{B}$.

\begin{figure}[!hbtp]
\includegraphics[angle=90, width=0.8\textwidth]{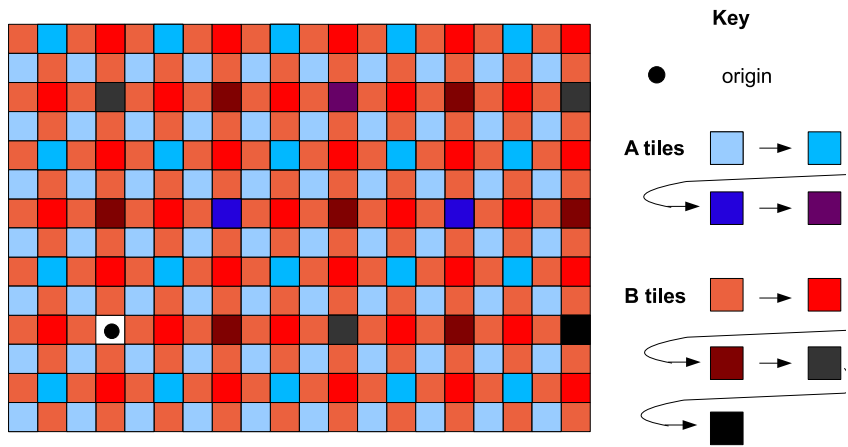}
\caption{Patch of $X(\cap A \cap B)^{\infty}$ \label{PartitionofXcapAcapBfinal}}
\end{figure}

\subsection{Illustrative example: Robinson tilings}

A Robinson tiling is a well known tiling from the field of aperiodic tilings \cite{ROBIN} \cite{GRUMSHEP}. A nonperiodic tiling is a tiling which is invariant only under a trivial translation. An aperiodic tiling is a nonperiodic tiling whose tiles cannot be rearranged into a periodic tiling.
A sample patch of a Robinson tiling can be found in figure \ref{robinsquare}.

\begin{figure}[!hbtp]
\includegraphics[angle=0, width=0.8\textwidth]{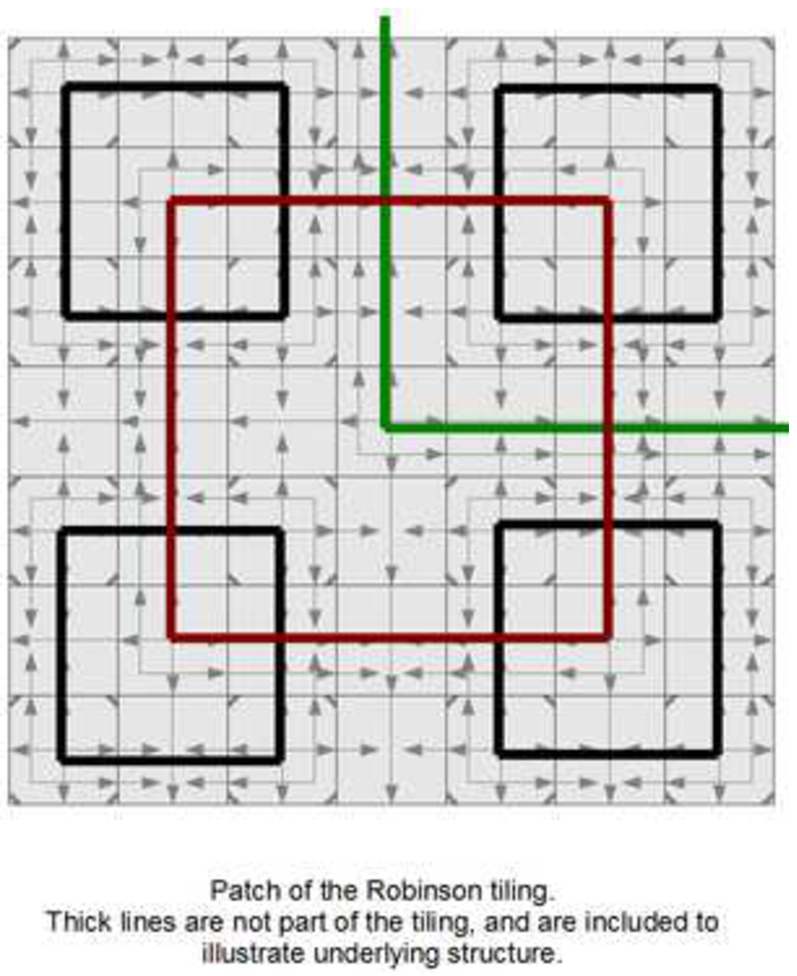}
\caption{Patch of the Robinson tiling \label{robinsquare}} 
\end{figure}

The defining characteristic of this tiling is a hierarchy of interlocking squares. In this tiling, four $3 \times 3$ squares form the four corners of a larger $7 \times 7$ square. Four of these $7 \times 7$ squares form the corners of a $15 \times 15$ square, which themselves form a $31 \times 31$ square, and so on.
This section will show how to construct the interlocking square structure via a two dimensional version of the $\cap$ function. (The structure is \lq Mutually Locally Derivable' to the Robinson tiling, as per the definition in \cite{SADUN}).
\\
\\
Let $B$ be a tiling by blank unit tiles.
Let $A$ be the periodic tiling depicted in figure \ref{robinsonperiodicB}, with $4$ prototiles with corner decorations. Note that the tile over the origin has a marking with a corner in the south west of the tile.

\begin{figure}[!hbtp]
\includegraphics[angle=0, width=0.8\textwidth]{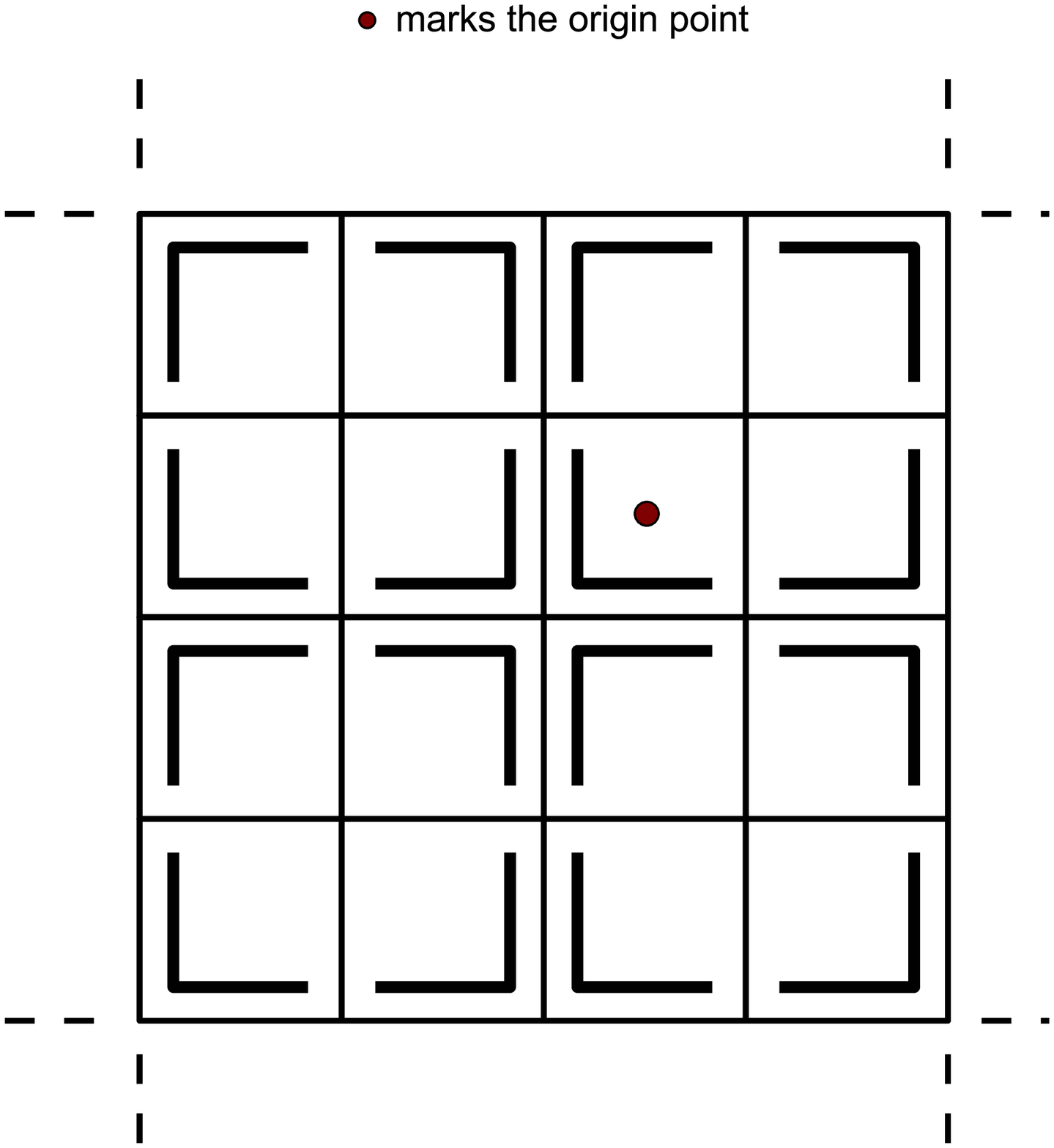}
\caption{$A$ \label{robinsonperiodicB}}
\end{figure}

Use $B$ as our seed tiling.
If we construct $B \cap A \cap B$, we produce the tiling depicted in figure \ref{RobBcapAcapBfinal}.

\begin{figure}[!hbtp]
\includegraphics[angle=90, width=0.8\textwidth]{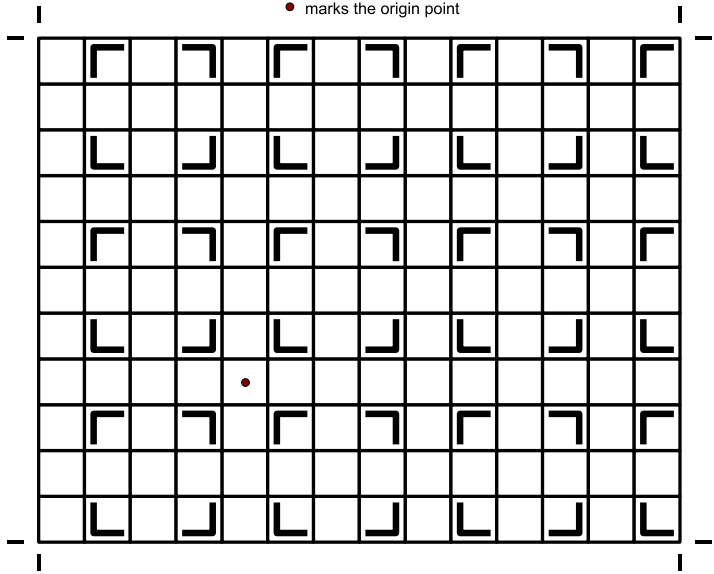}
\caption{$B (\cap A \cap B)$ \label{RobBcapAcapBfinal}}
\end{figure}

This pattern bears a certain similarity to the aperiodic Robinson tiling as shown in \ref{robinsquare}.
More precisely, the tiles from $A$ with corner decorations can be identified with cross tiles of the Robinson tiling. However only cross tiles which are corners of $3 \times 3$ tiles have a counterpart in $B \cap A \cap B$.

Let us consider $B( \cap A \cap B)^{2}$

\begin{figure}[!hbtp]
\includegraphics[angle=90, width=0.8\textwidth]{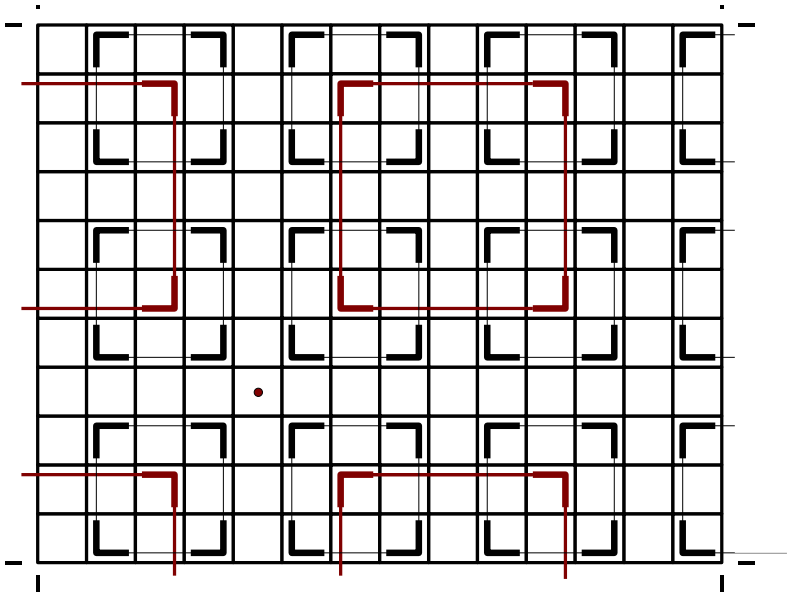}
\caption{$B (\cap A \cap B)^{2}$ \label{RobBcapAcapBcapAcapBalternate}}
\end{figure}

As the reader can see, applying $g_{B} \circ g_{A}$ again to $B \cap A \cap B$ adds another layer of corner tiles, corresponding to the cross tiles of $7 \times 7$ tiles.
This is due to the doubling property of the $g_{B} \circ g_{A}$ operation sending a $3 \times 3$ square onto a $7 \times 7$ square.
Similarly, $B( \cap A \cap B)^{3}$ will provide corner tiles corresponding to the cross tiles for the $15 \times 15$ squares, and in general $B( \cap A \cap B)^{n}$ will provide the cross tiles for all  $2^{n} \times 2^{n}$ squares (as well as cross tiles from smaller squares).

Since $B(\cap A \cap B)^{\infty}$ is well-defined, in the limit you get a tiling with corner tiles corresponding to every cross tile in a Robinson tiling. Since the cross tiles encode the positions of all $2^{m} \times 2^{m}$ squares (for $m \in \mathbb{Z}$) in the Robinson tiling, this is enough to encode the specific Robinson tiling. Thus $B(\cap A \cap B)^{\infty}$ is MLD to a Robinson tiling.

To be precise, it is MLD to a Robinson tiling with four faultlines meeting at the origin.

\begin{remark}
Other aperiodic tilings similar to the above Robinson tiling can be generated by changing the origin point of $g_{B} \circ g_{A}$. For further details, see \cite{DFLETCH}.
\end{remark}

\begin{remark}
The two dimensional variation of the $\cap$ operation bears similarities to certain areas of research on square tilings. In particular this concept may have uses with the Winfree model \cite{WINFREE} (self-assemblying constructions based on square tiles), since we have seen that the method can construct at least one two-dimensional structure.

Another possible use for the $\cap$ operation is in the field of computer graphics. Methods to create computer graphics from Wang tilings are well studied and form one of the most common methods of texture generation (for example \cite{MICRO}).
Briefly, a simple texture is placed onto each Wang prototile, and a tiling is formed from the Wang prototiles. The eye cannot see any periodicity in the tiling, thus a convincing graphic can be produced from simple textures.
This method breaks down when the graphic is approached, since the original simple textures can be seen more clearly, and pixelation may occur. Methods to stop this problem are currently being investigated (for example \cite{SIGGRAPH}).

The two dimensional $\cap$ function lends itself to this problem, since if you remove the rescaling factor of the $\cap$ function, the size of the tiles shrinks with every iteration. Thus by applying the $\cap$ function multiple times as the graphic is approached, new tiles can be introduced, and thus pixelation should be avoided.

On a more mathematical note, the last two chapters have described operations from limits of $\cap$ operations to Toeplitz sequences. Investigating whether a reverse operation can be constructed from a subset of Toeplitz sequences to compositions of $\cap$ operations would be an interesting avenue of inquiry.
\end{remark}

\end{document}